\documentclass{amsart}
\usepackage{amssymb}
\usepackage{amscd}
\usepackage{verbatim}
\usepackage{amsmath}
\usepackage{epsfig}
\begin{document}
\newcommand\Mand{\ \text{and}\ }
\newcommand\Mfor{\ \text{for}\ }
\newcommand\Real{\mathbb{R}}
\newcommand\RR{\mathbb{R}}
\newcommand\sphere{\mathbb{S}}
\newcommand\codim{\operatorname{codim}}
\newcommand\Sym{\operatorname{Sym}}
\newcommand\End{\operatorname{End}}
\newcommand\ep{\epsilon}
\newcommand\Cinf{C^\infty}
\newcommand\dCinf{\dot C^\infty}
\newcommand\CI{C^\infty}
\newcommand\dCI{\dot C^\infty}
\newcommand\Cx{\mathbb{C}}
\newcommand\Nat{\mathbb{N}}
\newcommand\dist{C^{-\infty}}
\newcommand\ddist{\dot C^{-\infty}}
\newcommand\pa{\partial}
\newcommand\Card{\mathrm{Card}}
\renewcommand\Box{{\square}}
\newcommand\WF{\mathrm{WF}}
\newcommand\WFb{\mathrm{WF}_\bl}
\newcommand\Vf{\mathcal{V}}
\newcommand\Vb{\mathcal{V}_\bl}
\newcommand\Hom{\mathrm{Hom}}
\newcommand\Id{\mathrm{Id}}
\newcommand\sgn{\operatorname{sgn}}
\newcommand\ff{\mathrm{ff}}
\newcommand\supp{\operatorname{supp}}
\newcommand\vol{\mathrm{vol}}
\newcommand\Diff{\mathrm{Diff}}
\newcommand\Diffd{\mathrm{Diff}_{\dagger}}
\newcommand\Diffs{\mathrm{Diff}_{\sharp}}
\newcommand\Diffb{\mathrm{Diff}_\bl}
\newcommand\Psib{\Psi_\bl}
\newcommand\Psibc{\Psi_{\mathrm{bc}}}
\newcommand\Tb{{}^{\bl} T}
\newcommand\Sb{{}^{\bl} S}
\newcommand\dom{\mathcal{D}}
\newcommand\cA{\mathcal{A}}
\newcommand\cE{\mathcal{E}}
\newcommand\cG{\mathcal{G}}
\newcommand\cH{\mathcal{H}}
\newcommand\cU{\mathcal{U}}
\newcommand\cO{\mathcal{O}}
\newcommand\cF{\mathcal{F}}
\newcommand\cM{\mathcal{M}}
\newcommand\cQ{\mathcal{Q}}
\newcommand\cR{\mathcal{R}}
\newcommand\cI{\mathcal{I}}
\newcommand\cL{\mathcal{L}}
\newcommand\Ptil{\tilde P}
\newcommand\ptil{\tilde p}
\newcommand\yt{\tilde y}
\newcommand\zetat{\tilde \zeta}
\newcommand\xit{\tilde \xi}
\newcommand\sigmah{\hat\sigma}
\newcommand\zetah{\hat\zeta}
\newcommand\loc{\mathrm{loc}}
\newcommand\compl{\mathrm{comp}}
\newcommand\reg{\mathrm{reg}}
\newcommand\GBB{\textsf{GBB}}
\newcommand\GBBsp{\textsf{GBB}\ }
\newcommand\bl{{\mathrm b}}
\newcommand{\sH}{\mathsf{H}}
\newcommand{\cte}{\digamma}
\newcommand\cl{\operatorname{cl}}
\newcommand\hsf{\mathcal{S}}

\setcounter{secnumdepth}{3}
\newtheorem{lemma}{Lemma}[section]
\newtheorem{prop}[lemma]{Proposition}
\newtheorem{thm}[lemma]{Theorem}
\newtheorem{cor}[lemma]{Corollary}
\newtheorem{result}[lemma]{Result}
\newtheorem*{thm*}{Theorem}
\newtheorem*{prop*}{Proposition}
\newtheorem*{cor*}{Corollary}
\newtheorem*{conj*}{Conjecture}
\numberwithin{equation}{section}
\theoremstyle{remark}
\newtheorem{rem}[lemma]{Remark}
\newtheorem*{rem*}{Remark}
\theoremstyle{definition}
\newtheorem{Def}[lemma]{Definition}
\newtheorem*{Def*}{Definition}

\title{Diffraction at corners for the wave equation on differential forms}
\author{Andr\'as Vasy}
\date{June 15, 2009}
\address{Department of Mathematics, Stanford University, CA 94305, USA}
\email{andras@math.stanford.edu}
\thanks{This work is partially supported by NSF grant DMS-0733485 and
DMS-0801226, and
a Chambers Fellowship from Stanford University.}

\begin{abstract}
In this paper we prove the propagation of singularities for the
wave equation on differential forms with natural (i.e.\ relative
or absolute) boundary conditions
on Lorentzian manifolds with corners,
which in particular includes a formulation of
Maxwell's equations.
These results are analogous to those obtained
by the author for the scalar wave equation \cite{Vasy:Propagation-Wave}
and for the wave equation on systems with Dirichlet or Neumann
boundary conditions in \cite{Vasy:Diffraction-edges}.
The main novelty is thus the presence of natural
boundary conditions, which effectively make the problem non-scalar,
even `to leading order',
at corners of codimension $\geq 2$.
\end{abstract}

\maketitle

\section{Introduction}
Let $X$ be a $\CI$ manifold with corners of dimension $n$ (a notion we
recall below),
and suppose that $h$ is a
Lorentz metric on $X$ of signature $(1,n-1)$ with dual metric $H$.
Thus, for $p\in X$, $H$ gives a non-degenerate symmetric bilinear form on
$T_p^*X$, which however is not positive definite. Then $H$ induces
a symmetric bilinear form on the real form bundle as usual (again, not
positive definite), and thus a Hermitian symmetric bilinear
form on the complex form bundle
which we write as $(.,.)_H$, or simply $(.,.)$.
Moreover, $h$ gives rise to a non-vanishing
density $|dh|$ on $X$; non-vanishing is the consequence of nondegeneracy.
In particular, for smooth forms, one of which has compact support,
one has a pairing
$$
\langle \alpha,\beta\rangle_H=\int_X (\alpha,\beta)_H \,|dh|.
$$

We recall here that a {\em tied (or t-)
manifold with corners} $X$ of dimension $n$ is a paracompact Hausdorff
topological space with a $\Cinf$ structure with corners, i.e.\ such
that the local coordinate charts map into
$[0,\infty)^k\times\Real^{n-k}$ rather than into $\Real^n$. Here $k$ varies
with the coordinate chart. We write $\pa_\ell X$ for the set of
points $p\in X$ such that in any local coordinates $\phi
=(\phi_1,\ldots,\phi_k,\phi_{k+1},\ldots,\phi_n)$ near $p$, with $k$ as above,
precisely
$\ell$ of the first $k$ coordinate functions vanish at $\phi(p)$.
We usually write such local coordinates as
$(x_1,\ldots,x_k,y_1,\ldots,y_{n-k})$.
A {\em boundary face} $F$
of codimension $\ell$ is the closure of a connected component
of $\pa_\ell X$. We write $F_{\reg}$ for the interior of $F$; note that
$F$ is a manifold with corners, while $F_{\reg}$ is a manifold
without boundary.
A boundary face of codimension $1$ is called a {\em boundary
hypersurface}. 
A {\em manifold with corners} is a tied manifold
with corners such that all boundary hypersurfaces are embedded submanifolds
-- as our results are local, this distinction is inessential.

{\em Throughout this paper we assume that every proper boundary
face $F$ of $X$ (i.e.\ all boundary faces but $X$ itself)
is time-like, in the sense that $H$ restricted to the
conormal bundle $N^*F$ of $F$ is negative definite.}

For $\CI$ differential forms on $X$, i.e.\ elements of $\CI(X;\Lambda X)$,
natural boundary conditions at the
boundary hypersurfaces $\hsf\in\pa_1(X)$ with non-vanishing
conormal $\nu_\hsf$ are $\nu_\hsf\wedge u=0$
(relative boundary condition) and $\iota_{\nu_\hsf} u=0$
(absolute boundary condition), and we write
\begin{equation*}\begin{split}
\CI_R(X;\Lambda X)=\{u\in\CI(X;\Lambda X):\ \forall \hsf\in\pa_1(X),
\ \nu_\hsf\wedge u|_\hsf=0\},\\
\CI_A(X;\Lambda X)=\{u\in\CI(X;\Lambda X):\ \forall \hsf\in\pa_1(X),
\ \iota_{\nu_\hsf} u|_\hsf=0\}.
\end{split}\end{equation*}
For $s\geq 0$ integer we let $H^s(X;\Lambda X)$ be the completion of
$\CI(X;\Lambda X)$ in the $H^s(X;\Lambda X)$ norm (defined up to
equivalence of norms on compact sets). The restriction map
$\CI(X;\Lambda X)\to\CI(\hsf;\Lambda_{\hsf} X)$ extends by continuity
to $H^1(X;\Lambda X)\to H^{1/2}(\hsf;\Lambda_\hsf X)$, and we let
\begin{equation*}\begin{split}
H^1_R(X;\Lambda X)=\{u\in H^1(X;\Lambda X):\ \forall \hsf\in\pa_1(X),
\ \nu_\hsf\wedge u=0\},\\
H^1_A(X;\Lambda X)=\{u\in H^1(X;\Lambda X):\ \forall \hsf\in\pa_1(X),
\ \iota_{\nu_\hsf} u=0\};
\end{split}\end{equation*}
these Sobolev spaces are also the closure of $\CI_R(X;\Lambda X)$,
resp.\ $\CI_A(X;\Lambda X)$, in $H^1(X;\Lambda X)$.

We consider the
wave equation $\Box u=f$ on $X$, where $\Box=\Box_h$ is the
d'Alembertian of $h$, with natural boundary conditions. That is,
for relative boundary conditions,
\begin{equation}\label{eq:Box-as-op}
\Box:H^1_{R,\loc}(X;\Lambda X)\to (H^1_{R,\compl}(X;\Lambda X))^*,
\end{equation}
and
for $u\in H^1_{R,\loc}(X;\Lambda X)$, $\Box u$ is given by
\begin{equation}\label{eq:Box-def}
\langle \Box u,v\rangle=\langle du,dv\rangle+\langle \delta u, \delta v\rangle,
\ v\in H^1_{R,\compl}(X;\Lambda X).
\end{equation}
That this is indeed the `right' boundary condition, e.g.\ in the
sense that if $X=M\times\RR_t$, $M$ Riemannian with metric $g$, $h=dt^2-g$,
and one considers the Laplacian on $M$ with relative boundary conditions,
and then the functional analytic solutions of $\Box u=f$ with this
boundary condition, follows from \cite{Mitrea-Taylor-Vasy:Lipschitz}, namely
that the quadratic form domain of $\Delta$ is $H^1_R(M;\Lambda M)$ in this
case -- see \cite[Section~2]{Vasy:Diffraction-edges} for more details.

The Hodge star $*$ maps $H^1_R(X;\Lambda X)$ to $H^1_A(X;\Lambda X)$
and conversely, and it intertwines solutions of $\Box u=f$ with relative
and absolute boundary conditions (of course, one has to take $*f$ as
well). Thus, it suffices to study relative boundary conditions, which
is how we proceed throughout this paper.

We recall that the analysis of singularities of solutions of the wave
equation takes place on the b-cotangent
bundle $\Tb^*X$ or on the b-cosphere bundle $\Sb^*X$.
Smooth sections of $\Tb^*X$ have the form
\begin{equation}\label{eq:b-sections}
\sum_{j=1}^k \sigma_j(x,y)\,\frac{dx_j}{x_j}+\sum_{j=1}^{n-k}
\zeta_j(x,y)\,dy_j,
\end{equation}
with $\sigma_j$ and $\zeta_j$ being $\CI$,
and correspondingly $(x,y,\sigma,\zeta)$ are local coordinates on $\Tb^*X$.
Let $o$ denote the zero section of $\Tb^*X$ (as well as other related
vector bundles below). Then $\Tb^*X\setminus o$ is equipped with
an $\Real^+$-action (fiberwise multiplication) which has no fixed
points. It is often natural to take the
quotient with the $\Real^+$-action, and work on the b-cosphere bundle,
$\Sb^*X$. In a region where, say,
\begin{equation}\label{eq:zeta-large}
|\sigma_j|<C|\zeta_{n-k}|,\ j=1,\ldots,k,\ |\zeta_j|<C|\zeta_{n-k}|,
j=1,\ldots,n-k-1,
\end{equation}
$C>0$ fixed, we can take
\begin{equation*}\begin{split}
&x_1,\ldots,x_k,y_1,\ldots,y_{n-k},\sigmah_1,\ldots,\sigmah_k,\zetah_1,
\ldots,\zetah_{n-k-1},|\zeta_{n-k}|,\\
&\sigmah_j=\frac{\sigma_j}{|\zeta_{n-k}|},
\ \zetah_j=\frac{\zeta_j}{|\zeta_{n-k}|},
\end{split}\end{equation*}
as (projective) local coordinates on $\Tb^*X\setminus o$, hence
$$
x_1,\ldots,x_k,y_1,\ldots,y_{n-k},\sigmah_1,\ldots,\sigmah_k,\zetah_1,
\ldots,\zetah_{n-k-1}
$$
as local coordinates on the image of this region under the quotient
map in $\Sb^*X$.

A somewhat different perspective is gained by considering the
dual bundle, $\Tb X$, of $\Tb^* X$. Locally its smooth sections have the
form
\begin{equation}\label{eq:b-vf}
\sum_{j=1}^k a_j (x_j\pa_{x_j})+\sum_{j=1}^{n-k} b_j \pa_{y_j},
\end{equation}
with $a_j,b_j\in\CI(X)$, corresponding to \eqref{eq:b-sections}.
Thus, these are exactly the $\CI$ vector fields on $X$ which are tangent
to every boundary face: they annihilate $x_j$ at $x_j=0$. The space
of these vector fields is denoted $\Vb(X)$, and the corresponding
differential operator algebra (locally finite sum of finite products
of elements of $\Vb(X)$) is $\Diffb(X)$.

The principal symbol of $\Box\in\Diff^2(X;\Lambda X)$ is $p\,\Id$, where
$p$ is the dual
metric function of $H$ on $T^*X$ (so $p(\alpha)=H(\alpha,\alpha)$),
and we denote the characteristic set of
$\Box$ by
$$
\Sigma=p^{-1}(\{0\})=\{q\in T^*X\setminus o:\ p(q)=0\}.
$$
We denote the Hamilton vector field of $p$ (on $T^*X$) by $\sH_p$.
There is a natural map $\pi:T^*X\to\Tb^*X$ induced by the corresponding
map between sections
$$
\sum_{j=1}^k \xi_j\, dx_j+\sum_{j=1}^{n-k} \zeta_j\, dy_j=\sum_{j=1}^k (x_j\xi_j)\,\frac{dx_j}{x_j}
+\sum_{j=1}^{n-k} \zeta_j\,dy_j,
$$
thus
\begin{equation}\label{eq:pi-in-coords}
\pi(x,y,\xi,\zeta)=(x,y,x\xi,\zeta),\ x\xi=(x_1\xi_1,\ldots,x_k\xi_k).
\end{equation}
We denote the image of
$\Sigma$ under $\pi$ by
$$
\dot\Sigma=\pi(\Sigma),
$$
called the compressed characteristic set.
As we show below in Section~\ref{sec:idea},
our assumptions on the time-like nature of every boundary
face $F$ imply that a neighborhood of $\dot\Sigma$ is covered by coordinate
charts as in \eqref{eq:zeta-large}, if the $\zeta_j$ are appropriately
numbered.
We next define generalized
broken bicharacteristics.

\begin{Def}\label{def:gen-br-bich}
{\em Generalized broken bicharacteristics,} or \GBB, are
continuous maps $\gamma:I\to\dot\Sigma$, where $I$ is an interval, satisfying
\begin{enumerate}
\item
for all $f\in\Cinf(\Tb^*X)$ real valued,
\begin{equation*}\begin{split}
&\liminf_{s\to s_0}\frac{(f\circ\gamma)(s)-(f\circ\gamma)(s_0)}{s-s_0}\\
&\ \geq \inf\{\sH_p(\pi^* f)(q):
\ q\in\pi^{-1}(\gamma(s_0))\cap\Sigma(P)\},
\end{split}\end{equation*}
\item
and if $q_0=\gamma(s_0)\in\Tb^*_{p_0}X$, and $p_0$ lies in the interior
of a boundary
hypersurface (i.e.\ a boundary face which has codimension $1$, so near $p_0$,
$\pa X$ is smooth),
then in a neighborhood of $s_0$,
$\gamma$ is a generalized broken bicharacteristic in the sense
of Melrose-Sj\"ostrand \cite{Melrose-Sjostrand:I}, see also
\cite[Definition~24.3.7]{Hor}.
\end{enumerate}
\end{Def}

We mention that there is a different (more concrete) but equivalent
version of this definition, due to
Lebeau \cite{Lebeau:Propagation}, which was used in
\cite{Vasy:Propagation-Wave}; we describe this in the next section and
refer to \cite{Vasy:Propagation-EDP} for a discussion of the relationship.

We next need to recall the definition of the b-wave front set, $\WFb(u)$,
which was introduced by Melrose originally in order to study
propagation of singularities on manifolds with smooth boundaries
\cite{Melrose:Transformation}.
More precisely, what we need
is the wave front set relative to $H^1(X;\Lambda X)$
rather than the more usual $L^2(X;\Lambda X)$, although they are equivalent
(with an appropriate shift of orders) for solutions of the wave
equation by Lemma~\ref{lemma:Dirichlet-form}, Proposition~\ref{prop:elliptic}
and the argument of
\cite[Lemma~6.1]{Vasy:Propagation-Wave} -- see
\cite{Vasy:Propagation-Wave} and \cite{Vasy:Diffraction-edges}
for a discussion.

As a first step we recall the space of the b-pseudodifferential operators
(b-ps.d.o's) which perform the required microlocalization.
There are two closely related
pseudodifferential algebras,
corresponding to the classical and symbolic algebras, $\Psi_{\cl}(X)$ and
$\Psi(X)$, in the boundaryless case.
These are denoted
by $\Psib(X)$ and $\Psibc(X)$, respectively.
There is also a principal symbol on $\Psib^m(X)$;
this is now a homogeneous degree $m$ function on $\Tb^*X\setminus o$.
$\Psib(X)$ has the algebraic properties analogous to $\Psi(X)$ on
manifolds without boundary. $\Psib(X)$ can be described quite explicitly;
this was done for instance
in \cite{Melrose-Piazza:Analytic, Vasy:Propagation-Wave, Vasy:Diffraction-edges}
in the corners setting,
and in \cite[Section~18.3]{Hor} for smooth boundaries. In particular,
a {\em subset} of $\Psibc(X)$ (which would morally suffice for our purposes
here) consists of operators with Schwartz kernels
supported in $U\times U$, $U\subset X$ a coordinate chart with coordinates
$x,y$ as above, with Schwartz kernels of the form
\begin{equation}\begin{split}\label{eq:b-quantize}
&q(a)u(x,y)\\
&\qquad=(2\pi)^{-n}\int e^{i((x-x')\cdot \xi+(y-y')\cdot\zeta)}
\phi(\frac{x-x'}{x})
a(x,y,x\xi,\zeta) u(x',y')\,dx'\,dy'\,d\xi\,d\zeta,
\end{split}\end{equation}
understood as an oscillatory integral,
where $a\in S^m(\RR^n_{x,y};\RR^n_{\sigma,\zeta})$ (with $\sigma=x\xi$,
cf.\ \eqref{eq:pi-in-coords}),
$\phi\in\Cinf_{\compl}((-1/2,1/2)^k)$ is identically $1$ near $0$,
$\frac{x-x'}{x}=(\frac{x_1-x_1'}{x_1},\ldots,\frac{x_k-x_k'}{x_k})$,
and the integral in $x'$ is over $[0,\infty)^k$. This formula is similar
to the standard quantization formula, but $\xi$ is replaced by $x\xi$ here
in the argument of $a$, and there is a localizating factor $\phi$
which being identically $1$ near the diagonal, does not play an important
role. A subset of $\Psib(X)$ is similarly obtained if we require
that $a$ is a classical (i.e.\ one-step polyhomogeneous) symbol.
Thus, if $a$ is a polynomial in its third and fourth slots,
i.e.\ in $x\xi$ and $\zeta$, depending smoothly
on $x,y$, i.e.
$$
a(x,y,\xi,\zeta)=\sum_{|\alpha|+|\beta|\leq m}
a_{\alpha\beta}(x,y) (x\xi)^\alpha\zeta^\beta,
$$
then
$$
q(a)=\sum_{|\alpha|+|\beta|\leq m}
a_{\alpha\beta}(x,y) (xD_x)^\alpha D_y^\beta,
$$
thus connecting $\Vb(X)$ and $\Diffb(X)$ to $\Psib(X)$ in view of
\eqref{eq:b-vf}.
For vector bundles $E,F$ over $X$, one can also construct
$\Psib(X;E,F)$ via trivializations -- acting between distributional sections
of vector bundles $E$ and $F$ over $X$. Elements of $\Psibc^m(X)$ have
the important property that they map $\CI(X)\to\CI(X)$, and more
generally they map $x_j\CI(X)\to x_j\CI(X)$, so if $A\in\Psibc^m(X)$, then
$(Au)|_{\hsf_j}$ depends only on $u|_{\hsf_j}$ for $u\in\CI(X)$. In particular,
Dirichlet boundary conditions are automatically preserved by such $A$,
which makes $\Psib(X)$ easy to use in the analysis of the Dirichlet
problem in \cite{Vasy:Propagation-Wave}. We will need more care
for natural boundary conditions,
which is a point we address in the next section, see also
\cite{Vasy:Diffraction-edges}.

The space of `very nice' functions corresponding
to $\Vb(X)$ and $\Diffb(X)$, replacing $\Cinf(X)$, is the space of
conormal functions to the boundary relative to a fixed space of functions,
in this case $H^1(X;\Lambda X)$, i.e.\ functions
$v\in H^1_{\loc}(X;\Lambda X)$
such that $Qv\in H^1_{\loc}(X;\Lambda X)$ for every
$Q\in\Diffb(X;\Lambda X)$ (of any order).
Then $q\in\Tb^*X\setminus o$
is {\em not} in $\WFb(u)$ if there is an $A\in\Psib^0(X;\Lambda X)$ such that
$\sigma_{\bl,0}(A)(q)$ is invertible and $Au$ is $H^1$-conormal to the boundary.
Spelling out the latter explicitly, and also defining the wave front
set with finite regularity:

\begin{Def}
Suppose $u\in H^1_{\loc}(X;\Lambda X)$. Then $q\in\Tb^*X\setminus o$
is {\em not} in $\WFb^{1,\infty}(u)$ if there is an
$A\in\Psib^0(X;\Lambda X)$ such that
$\sigma_{\bl,0}(A)(q)$ is invertible and
$QAu\in H^1_{\loc}(X;\Lambda X)$ for all
$Q\in\Diffb(X;\Lambda X)$.

Moreover, $q\in\Tb^*X\setminus o$
is {\em not} in $\WFb^{1,m}(u)$ if there is an $A\in\Psib^m(X;\Lambda X)$
such that
$\sigma_{\bl,0}(A)(q)$ is invertible and $Au\in H^1_{\loc}(X;\Lambda X)$.
\end{Def}

The wave front set relative to the dual space, $\dot H^{-1}_R(X;\Lambda X)$,
is defined similarly.

We recall that the definition of $\WF$ could be stated in a completely
parallel manner: we would require (for $X$ without boundary)
$QAu\in L^2(X)$ for all $Q\in\Diff(X)$ -- this is equivalent to
$Au\in\Cinf(X)$ by the Sobolev embedding theorem. Here $L^2(X)$ can be replaced
by $H^m(X)$ instead, with $m$ arbitrary. Similarly, $\WF^m$ can also
be defined analogously: we require $Au\in L^2(X)$ for $A\in\Psi^m(X)$
elliptic at $q$.

The usefulness of the definition relies on the fact that
any $A\in\Psibc^0(X;\Lambda X)$ with compact support defines a continuous linear
maps $A:H^1(X;\Lambda X)\to H^1(X;\Lambda X)$
with norms bounded by a seminorm of $A$
in $\Psibc^0(X;\Lambda X)$, see \cite[Lemma~3.2]{Vasy:Propagation-Wave}
in the scalar case and the discussion after
\cite[Definition~6]{Vasy:Diffraction-edges} for the vector-valued case.

Our main result is the following:

\begin{thm}[See \cite{Vasy:Propagation-Wave} for the scalar equation
if $X=M\times\RR$ with a product metric, and \cite{Vasy:Diffraction-edges}
for the vector-valued equation with Dirichlet or Neumann boundary conditions,
and see Theorem~\ref{thm:prop-sing-restate} for a strengthened
restatement.]
\label{thm:prop-sing}
Suppose that $X$ is a manifold with corners with a $\CI$
Lorentz metric $h$ with respect to which every boundary face is
timelike.
Suppose $u\in H^1_{R,\loc}(X;\Lambda X)$ and
$\Box u=f$ in the sense of \eqref{eq:Box-as-op}-\eqref{eq:Box-def}
holding for all $v\in H^1_{R,\compl}(X;\Lambda X)$.
Then
$$
(\WFb^{1,m}(u)\cap\dot\Sigma)\setminus\WFb^{-1,m+1}(f)
$$
is a union of maximally extended generalized broken bicharacteristics
of $\Box$ in
$$
\dot\Sigma\setminus\WFb^{-1,m+1}(f).
$$

In particular, if $\Box u=0$ then $\WFb^{1,\infty}(u)\subset\dot\Sigma$
is a union of maximally extended generalized broken bicharacteristics
of $\Box$.

The same results hold with the relative boundary condition ($R$) replaced
by the absolute boundary condition ($A$) throughout.
\end{thm}

On manifolds with $\CI$ boundaries (and no corners),
in the scalar setting this result is due to
Melrose, Sj\"ostrand and
Taylor \cite{Melrose-Sjostrand:I, Melrose-Sjostrand:II,
Taylor:Grazing, Melrose-Taylor:Kirchhoff}.
Still for $\CI$ boundaries and singularities, but
for systems,
including differential forms with natural boundary conditions,
Taylor, and Melrose and Taylor have shown the theorem
at diffractive
points via a parametrix construction \cite{Taylor:Grazing-II,
Melrose-Taylor:Boundary}, see also Yingst's work \cite{Yingst:Kirchhoff}.
In addition, Ivrii \cite{Ivrii:Wave}
has obtained propagation results for systems. Thus, the theorem
is known for $\CI$ boundaries and $\CI$ singularities.

The analogue of this theorem for the scalar equation but for analytic
singularities and spaces with an analytically stratified boundary
was proved by Lebeau \cite{Lebeau:Propagation},
following the work of Sj\"ostrand \cite{Sjostrand:Propagation-I} when
the boundary is analytic. As far as the author is aware, there is
no known analogue of this result in the analytic setting for systems
with natural boundary conditions, including gliding rays.

A special case of the scalar equation with codimension 2 corners
in $\Real^2$ had been considered by P.~G\'erard and Lebeau
\cite{Gerard-Lebeau:Diffusion} in the real analytic setting, and by
Ivrii \cite{Ivrii:Propagation-Corners} in the smooth setting.
It should also be mentioned that
due to its relevance, this problem has a long history, and has been studied
extensively by Keller in the 1940s and 1950s in various special settings,
see e.g.\ \cite{Blank-Keller:Diffraction, Keller:Diffraction}.

It is an interesting question to what extent the particular non-scalar
boundary conditions we have chosen (namely, relative or absolute) matter,
i.e.\ whether the results would hold for $\Box$, or similar operators
$P$ (as in \eqref{eq:P-form}, but acting on different spaces),
with other non-scalar boundary conditions, and also whether the form
bundle can be replaced by other bundles. We address this issue
in the last section of the paper.

The structure of this paper is the following. In Section~\ref{sec:idea}
we describe the geometric background in more detail, and use this to
explain the idea of the proof. In Section~\ref{sec:commutator}
we recall some commutator calculation preliminaries and prove the
main `commutator' lemma that we use later in the paper. In
Section~\ref{sec:elliptic} we recall the elliptic results
from \cite{Vasy:Propagation-Wave} and \cite{Vasy:Diffraction-edges}. Then in
Section~\ref{sec:normal} we prove the normal propagation estimate, and in
Section~\ref{sec:tangential} we prove glancing
propagation. In Section~\ref{sec:prop-sing}
we put together these results, and also extend the results to a larger class
of solutions, possessing a negative order of $\bl$-regularity relative to
$H^1_{R,\loc}(X;\Lambda X)$. Finally, in Section~\ref{sec:other}
we discuss other vector bundles and boundary conditions for which the
result holds.

I am very grateful to Rafe Mazzeo, Richard Melrose, Michael Taylor,
Gunther Uhlmann
and Jared Wunsch for
their interest in this project, for comments on the manuscript
and for fruitful discussions.

\section{Setup and idea of the proof}\label{sec:idea}

In order to explain how the theorem is proved,
we describe the geometry more precisely.

First, we choose local coordinates more carefully. In arbitrary local
coordinates
$$
(x_1,\ldots,x_k,y_1,\ldots,y_{n-k})
$$
on a neighborhood $\cU$
of a point in the interior
of a codimension $k$ corner $F$ given by $x_1=\ldots=x_k=0$ inside
$x_1\geq 0,\ldots, x_k\geq 0$, any symmetric bilinear
form on $T^*X$ can be written
as
\begin{equation}\label{eq:metric-form2}
H(x,y)=\sum_{i,j}A_{ij}(x,y)\,\pa_{x_i}\,\pa_{x_j}
+\sum_{i,j}2 C_{ij}(x,y)\,\pa_{x_i}\,\pa_{y_j}+\sum_{i,j}
B_{ij}(x,y)\,\pa_{y_i}\,\pa_{y_j}
\end{equation}
with $A,B,C$ smooth. Below we write covectors as
\begin{equation}\label{eq:T-star-X-coords}
\alpha=\sum_{i=1}^k\xi_i\,dx_i+\sum_{i=1}^{n-k}\zeta_i\,dy_i.
\end{equation}
Since we assume that
every boundary face, in particular $F$, is time-like in the sense
that the restriction of $H$ to $N^*F$ is negative definite, we deduce that
$A$ is negative definite, for locally the conormal bundle $N^*F$ is given by
$$
\{(x,y,\xi,\zeta):\ x=0,\ \zeta=0\}.
$$
Then $H$ is Lorentzian on the $H$-orthocomplement
$(N^*F)^\perp$ of $N^*F$. In fact, note that for $p_0\in F$,
\begin{equation}\label{eq:orth-decomp}
T^*_{p_0}X=N^*_{p_0}X\oplus (N^*_{p_0}X)^\perp,
\end{equation}
for if $V$ is
in the intersection of the two summands, then $H(V,V)=0$ and $V\in N^*_{p_0}F$,
so the definiteness of the inner product on $N^*F$ shows that $V=0$, hence
\eqref{eq:orth-decomp} follows as the dimension of the summands sums up to
the dimension of $T^*_{p_0}X$.
Choosing an orthogonal basis of $(N^*F)^\perp$ consisting of vectors
of length $\pm 1$ at a given point
$p_0\in F^\circ$, and then coordinates $y_j$ with differentials equal
to these basis vectors, we have in the new basis that $C_{ij}(0,0)=0$ and
\begin{equation}\label{eq:B-orth-0}
\sum B_{ij}(0,0)\pa_{y_i}\pa_{y_j}=\pa_{y_{n-k}}^2-\sum_{i<n-k}\pa_{y_i}^2,
\end{equation}
and we write coordinates on $T^*X$ as 
$$
x,\ t=y_{n-k},\ \yt=(y_1,\ldots,y_{n-k-1}),\ \xi,\ \tau=\zeta_{n-k},\ \zetat
=(\zeta_1,\ldots,\zeta_{n-k-1}),
$$
cf.\ \eqref{eq:T-star-X-coords}.
Thus $B$ is non-degenerate, Lorentzian,
near $p_0$, and a simple calculation shows that
the coordinates on $X$ can be chosen
(i.e.\ the $y_j$ can be adjusted) so
that $C(0,y)=0$.
Then
\begin{equation}\label{eq:H-at-corner2}
H|_{x=0}=\sum_{i,j}A_{ij}(0,y)\,\pa_{x_i}\,\pa_{x_j}
+\sum_{i,j}B_{ij}(0,y)\,\pa_{y_i}\,\pa_{y_j},
\end{equation}
and hence the metric function is
\begin{equation}\label{eq:p-at-corner2}
p|_{x=0}=\xi\cdot A(y)\xi+\zeta \cdot B(y)\zeta.
\end{equation}
This gives that
\begin{equation}\begin{split}\label{eq:dot-Sigma-coords2}
\dot\Sigma\cap\cU\cap\Tb^*_F X=\{(0,y,0,\zeta):\ 0\leq
\zeta \cdot B(y)\zeta,\ \zeta\neq 0\}.
\end{split}\end{equation}
In particular, in view of \eqref{eq:B-orth-0},
$\dot\Sigma\cap\cU$ lies in the region \eqref{eq:zeta-large}, at least
after we possibly shrink $\cU$.

In order to better understand the generalized broken bicharacteristics
for $\Box$,
we divide $\dot\Sigma$ into two subsets.
We thus define the {\em glancing set} $\cG$
as the set of points in $\dot\Sigma$ whose preimage under
$\hat\pi=\pi|_{\Sigma}$
consists of a single point, and define the {\em hyperbolic set} $\cH$
as its complement
in $\dot\Sigma$. Thus, $q\in\dot\Sigma$ lies in $\cG$ if and only
if on $\hat\pi^{-1}(\{q\})$, $\xi_j=0$ for all $j$.
More explicitly, with the notation of
\eqref{eq:dot-Sigma-coords2},
\begin{equation}\begin{split}\label{eq:H-G-exp2}
&\cG\cap\cU\cap\Tb^*_F X=\{(0,y,0,\zeta):\ \zeta \cdot B(y)\zeta=0,\ \zeta\neq 0\},\\
&\cH\cap\cU\cap\Tb^*_F X=\{(0,y,0,\zeta):\ \zeta \cdot B(y)\zeta>0,\ \zeta\neq 0\}.
\end{split}\end{equation}
Thus, $\cG$ corresponds to generalized broken bicharacteristics which
are tangent to $F$ in view of the vanishing of $\xi_j$, while
$\cH$ corresponds to generalized broken bicharacteristics which
are normal to $F$. Note that if $F$ is one-dimensional, which is the lowest
dimension it can be in view of the time-like restriction, then
$\zeta\cdot B(y)\zeta$ necessarily implies $\zeta=0$, so in fact
$\cG\cap\Tb^*_FX=\emptyset$.

We next make the role of $\cG$ and $\cH$ more explicit, which
explains the relevant phenomena better. An equivalent
characterization of \GBBsp is

\begin{lemma}(See the discussion in \cite[Section~1]{Vasy:Propagation-EDP}
after the statement of Definition~1.1.)\label{lemma:gen-br-bichar}
A continuous map
$\gamma:I\to\dot\Sigma$,
where $I\subset\Real$ is an interval, is a \GBBsp if and only if it satisfies
the following requirements:

\begin{enumerate}
\item
If $q_0=\gamma(s_0)\in\cG$
then for all
$f\in\Cinf(\Tb^*X)$,
\begin{equation}\label{eq:cG-bich}
\frac{d}{ds}(f\circ \gamma)(s_0)=\sH_p (\pi^*f)(\tilde q_0),\ \tilde q_0=\hat\pi^{-1}(q_0).
\end{equation}

\item
If $q_0=\gamma(s_0)\in\cH\cap \Tb^*_{F_\reg} X$ then
there exists $\ep>0$ such that
\begin{equation}
s\in I,\ 0<|s-s_0|<\ep\Rightarrow\gamma(t)\notin \Tb^*_{F_\reg} X.
\end{equation}

\item
If $q_0=\gamma(s_0)\in\cG\cap \Tb^*_{F_{\reg}}X$, and $F$ is a boundary
hypersurface (i.e.\ has codimension $1$), then in a neighborhood of $s_0$,
$\gamma$ is a generalized broken bicharacteristic in the sense
of Melrose-Sj\"ostrand \cite{Melrose-Sjostrand:I}, see also
\cite[Definition~24.3.7]{Hor}.
\end{enumerate}
\end{lemma}

The general strategy of the proof of the main theorem is to prove propagation
estimates at $\cG$ and $\cH$ separately. The estimates at $\cH$ can be weaker:
as the \GBBsp through these points are normal, one only needs to prove
that singularities leave $\Tb^*_{F_{\reg}} X$ immediately,
for then an inductive
argument, using that locally $F$ is the most singular stratum, allows one
to deduce the desired propagation.

However, at $\cG$ one has to prove a more precise result. Namely, if
$q_0\in\cG$, there is a unique point $\alpha_0\in\hat\pi^{-1}(\{q_0\})$,
and we need to prove that roughly speaking
singularities propagate in the direction of
$(\pi_*)_{\alpha_0}\sH_p$. More precisely (although we actually use a vector
field on $T^*F$ in Section~\ref{sec:tangential}, and a product decomposition
of $X$ near a point in $F_{\reg}$), let $W$ be a vector field on
$\Tb^*X$ with $W(q_0)=(\pi_*)_{\alpha_0}\sH_p$. Then we need to prove that
for small $\delta>0$, there
is an $o(\delta)$-sized ball around $\exp(\delta W)\alpha_0$ such that
if this ball contains no singularities of $u$, then $\alpha_0\notin
\WFb^{1,m}(u)$ either. We show that indeed
there is an $O(\delta^2)$-sized ball with this property,
just like for the scalar or the vector-valued
equation on manifolds with corners with Dirichlet or Neumann boundary
condition.

One basic difficulty is that
we need to
use operators which preserve the boundary conditions in order to microlocalize.
As we want to use principally scalar operators, at the principal symbol
level this is automatic, but we need operators fully (not merely
symbolically) preserving the boundary conditions. To achieve this,
we locally trivialize the form bundle (and use microlocalizers supported
in such a coordinate chart) in such a way that the boundary
conditions state the vanishing of various components of the trivialization.
More concretely,
a local trivialization over an open set $\cU$ of $\Lambda^p X$ is a map
$\Lambda^p_{\cU}X\to \cU\times \RR^N$, $N=\dim\Lambda^p_q X$, $q\in X$,
being given
by the binomial coefficient; we want this such that
there is an index set $J_j\subset\{1,\ldots,N\}$
for $j=1,\ldots,k$,
such that for each $j$ and
at each $q\in \cU\cap \hsf_j=\{x_j=0\}$, for a form
$\alpha$ to satisfy $dx_j\wedge u=0$ requires precisely that $\alpha_m=0$ for
$m\in J_j$, where $\alpha=(\alpha_1,\ldots,\alpha_N)$ with repect
to the trivialization. The construction of such a trivialization
is straightforward, however, using
\begin{equation}\label{eq:form-triv-corner}
dx_{i_1}\wedge\ldots\wedge
dx_{i_s}\wedge dy_{\ell_1}\wedge\ldots\wedge dy_{\ell_{p-s}},\ i_1<\ldots<i_s,
\ \ell_1<\ldots<\ell_{p-s},
\end{equation}
as the basis of $\Lambda^p_q X$, $dx_j\wedge u=0$ amounts to saying
that all components of $\alpha$ in which $j$ is not one the $i_r$'s
vanish. Similarly, using the Hodge star operator, there is
such a good trivialization for the absolute boundary condition as well,
namely $*$ applied to the basis of \eqref{eq:form-triv-corner}.

Now the mere existence of such a trivialization, and hence of `scalar'
(namely, diagonal with respect to the trivialization) b-pseudodifferential
operators guarantees that the elliptic regularity arguments go through
since in these arguments commutators are lower order hence negligible.
Thus, microlocal elliptic regularity was proved by the author
in \cite{Vasy:Diffraction-edges}. However, the matters are much more
complicated for hyperbolic and glancing points, as at these points
the propagation estimates are positive commutator estimates. In positive
commutator estimates it is convenient to have formally self-adjoint
commutants; however, if one has a scalar operator, its adjoint is
usually not scalar -- and indeed, usually does not preserve even the relevant
subbundles, hence the boundary
conditions\footnote{This is the difference with codimension one boundaries,
where
one can simply take an {\em orthogonal} decomposition of the cotangent
bundle $dx_1$ being orthogonal to the span of $dy_1,\ldots,dy_{n-1}$, which
implies that adjoints of block-diagonal operators,
i.e.\ operators respecting
this decomposition, are also block-diagonal at $\hsf=\hsf_1$, which is
what matters.}.
One can also work with non-self-adjoint commutants, in which
case one has in the boundaryless setting
$$
\langle \Box u,Au\rangle-\langle Au,\Box u\rangle
=\langle (A^*-A)\Box u,u\rangle+\langle[A,\Box]u,u\rangle;
$$
we refer to Proposition~\ref{prop:gend-comm} for the correct statement in the
presence of boundaries.
Roughly speaking,
the difficulty here is that $A^*-A$  does not preserve the natural
boundary conditions, and $(A^*-A)\Box$ is the same order as $[A,\Box]$
(if $A$ has real scalar principal symbol) -- this is a problem
since if $A^*-A$ preserved boundary conditions,
$\langle \Box u,(A^*-A)u\rangle$ would be controlled by the PDE, but this
is not so otherwise\footnote{Indeed, even the principal symbol of $A^*-A$
does not preserve boundary conditions typically.}.
However, it turns out that modulo terms one can easily
control (because normal derivatives of $u$ are small at glancing points),
one can replace $\Box$ by its tangential part, and instead of using
the PDE, use the positivity of the commutator to control this term.

Our results then combine to prove the main theorem, using the argument
of Melrose and Sj\"ostrand \cite{Melrose-Sjostrand:I,
Melrose-Sjostrand:II}, as modified by Lebeau
\cite[Proposition~VII.1]{Lebeau:Propagation}.

\section{Commutator constructions}\label{sec:commutator}
We start by recalling\footnote{See \cite{Melrose:Atiyah},
which mostly deals with the $\CI$ boundary case, and especially
\cite{Melrose-Piazza:Analytic} as a background reference;
\cite{Vasy:Propagation-Wave} has a brief discussion as well on
$\Psibc(X)$.}
that for any vector bundle $E$ over $X$,
$$
\Psibc(X;E)=\cup_s\Psibc^s(X;E),
\ \Psib(X;E)=\cup_s\Psib^s(X;E)\subset\Psibc(X;E),
$$
are sets of operators
\begin{equation}\begin{split}\label{eq:Psibc-CI}
&\Psibc(X;E)\ni A:\dCI(X;E)\to\dCI(X;E),\\
&\Psibc(X;E)\ni A:\CI(X;E)\to\CI(X;E),
\end{split}\end{equation}
where $\dCI(X;E)$ denotes the subspace of $\CI(X;E)$ ($\CI$ sections of
$E$) which vanish with all derivatives at $\pa X$. Then
$\Psibc(X;E)$ is a filtered algebra of operators, $\Psib(X;E)$ is
closed under composition, as well as under addition under a compatibility
condition on orders,
with a principal symbol
map
$$
\sigma_{\bl,s}:\Psib^s(X;E)\to S^s_{\hom}(\Tb^*X\setminus o;\pi^*\Hom(E,E)),
$$
where $\pi:\Tb^*X\to X$ is the bundle projection, and $S^s_{\hom}$ denotes
homogeneous degree $s$, $\CI$ functions on $\Tb^*X\setminus o$, while
$$
\sigma_{\bl,s}:\Psibc^s(X;E)\to S^s(\Tb^*X;\pi^*\Hom(E,E))
/S^{s-1}(\Tb^*X;\pi^*\Hom(E,E)).
$$
Thus,
if $B_j\in\Psib^{s_j}(X;E)$, $j=1,2$, then
$$
B_1 B_2\in\Psib^{s_1+s_2}(X;E),
$$
with
$$
\sigma_{\bl,s_1+s_2}(B_1 B_2)(q)=\sigma_{\bl,s_1}(B_1)(q)
\sigma_{\bl,s_2}(B_2)(q),\ q\in\Tb^*X\setminus o,
$$
where the product of the right is composition of endomorphisms of
the fiber of $E$ at the point $\pi(q)$, and similarly for $\Psibc^s(X;E)$.

If $B_j\in\Psib^{s_j}(X;E)$ and $\sigma_{\bl,s_j}(B_j)$ is
{\em scalar}, $j=1,2$, i.e.\ is a multiple of the identity homomorphism:
$$
\sigma_{\bl,s_j}(B_j)=b_j\,\Id,\ b_j\in S^{s_j}_{\hom}(\Tb^*X),
$$
then their commutator is
$$
[B_1,B_2]=B_1 B_2-B_2 B_1\in\Psib^{s_1+s_2-1}(X;E),
$$
with
$$
\sigma_{\bl,s_1+s_2-1}([B_1,B_2])=\imath(\sH_{\bl,b_1} b_2)\,\Id;
$$
the analogous result also holds for $\Psibc(X;E)$.
Here $\sH_{\bl,a}$ is the b-Hamilton vector field of $a\in\CI(\Tb^*X)$,
i.e.\ it is the unique $\CI$ vector field on $\Tb^*X$ which agrees
with the standard Hamilton vector field $\sH_{a}$ on $T^* X^\circ$ under
the natural identification of $T^*X^\circ$ with $\Tb^*_{X^\circ}X$. Thus,
in the notation of \eqref{eq:pi-in-coords},
$\sH_{\bl,a}=\pi_*\sH_{\pi^*a}$. In
local coordinates,
$$
\sH_{\bl,a}=\sum_j (\pa_{\sigma_j}a)\,x_j\pa_{x_j}+\sum_j(\pa_{\zeta_j}a)
\,\pa_{y_j}-\sum_j (x_j\pa_{x_j} a)\,\pa_{\sigma_j}-\sum_j (\pa_{y_j}a)
\,\pa_{\zeta_j},
$$
see e.g.\ \cite[Proof of Lemma~2.8]{Vasy:Propagation-Wave}.
In particular, it is a vector field tangent to all boundary faces
of $\Tb^* X$, and has vanishing $\pa_{\sigma_j}$ component at $\{x_j=0\}$.
Note also that $\sigma_{\bl,s_1+s_2-1}([B_1,B_2])$ depends only on
$b_1$ and $b_2$.

On the other hand, suppose now that $B_j\in\Psib^{s_j}(X;E)$, $j=1,2$, and
$\sigma_{\bl,s_1}(B_1)$ is scalar. Then $\sigma_{\bl,s_1}(B_1)$ and
$\sigma_{\bl,s_2}(B_2)$ commute, hence
$$
\sigma_{\bl,s_1+s_2}([B_1,B_2])=0,
$$
so
$$
[B_1,B_2]\in\Psib^{s_1+s_2-1}(X;E).
$$
However, the principal symbol of the commutator now depends on $B_1$
via more than its principal symbol -- see also Remark~\ref{rem:subpr-adj-dep}.

We also recall that if we equip $E$ with a Hermitian inner product
and put a $\CI$ density $\nu$ on $X$, thus obtaining an inner product
on $L^2(X;E)$, then $A\in\Psibc^0(X;E)$ with
compactly supported Schwartz kernel is bounded, with
norm bounded by a seminorm of $A$
in $\Psibc^0(X;E)$ -- see \cite[Equation~(2.16)]{Melrose-Piazza:Analytic}.
Moreover, if $A$ has scalar principal symbol
$\sigma_{\bl,0}(A)=a\,\Id$, then there
exists $A'\in\Psib^{-1}(X;E)$ such that for all
$v\in L^2(X;E)$,
\begin{equation*}
\|Av\|\leq 2\sup|a|\,\|v\|+\|A'v\|;
\end{equation*}
see \cite[Section~2]{Vasy:Propagation-Wave}.
In addition, $\Psibc^0(X;E)$ and $\Psib^0(X;E)$ are closed under
$L^2$-adjoints. Thus, dually -- with respect to the $L^2$-inner product --
to \eqref{eq:Psibc-CI}, with
$\dist(X;E)$ the dual of $\dCI(X;E)$, $\ddist(X;E)$ dual to $\CI(X;E)$,
\begin{equation}\begin{split}\label{eq:Psibc-dist}
&\Psibc(X;E)\ni A:\dist(X;E)\to\dist(X;E),\ \text{resp.}\\
&\Psibc(X;E)\ni A:\ddist(X;E)\to\ddist(X;E),
\end{split}\end{equation}
defined by
\begin{equation}\label{eq:Psibc-dist-def}
Au(\phi)=u(A^*\phi),\ \phi\in\dCI(X;E),\ \text{resp.}\ \phi\in\CI(X;E).
\end{equation}

Next, the definition of $\Diff\Psib(X)$ from
\cite[Definition~2.3]{Vasy:Propagation-Wave}:

\begin{Def}\label{Def:Diff-Psib}
$\Diff^k\Psib^s(X)$ is the vector space of operators of the form
\begin{equation}\label{eq:Diff-Psib-def}
\sum_j P_j A_j,\ P_j\in\Diff^k(X),\ A_j\in\Psib^s(X),
\end{equation}
where the sum is locally finite in $X$.
\end{Def}

Equivalently, the order of the factors can be reversed, i.e.\ these
operators can be written as
\begin{equation*}
\sum_j A'_j P'_j,\ P'_j\in\Diff^k(X),\ A'_j\in\Psib^s(X).
\end{equation*}
The key point (in local coordinates) is
that while $D_{x_j}\notin\Vb(X)$,
for any $A\in\Psib^{m}(X)$ there is an operator $\tilde A\in\Psib^m(X)$
such that
\begin{equation}\label{eq:V-Psib-1}
D_{x_j}A-\tilde A D_{x_j}\in\Psib^m(X),
\end{equation}
and analogously for $\Psib^m(X)$ replaced by $\Psibc^m(X)$, see
\cite[Equation~(2.3)]{Vasy:Propagation-Wave}.
Indeed, one may write
\begin{equation}\label{eq:A-At-Bt}
D_{x_j}A=\tilde A D_{x_j}+\tilde B,\ \tilde A=x_j^{-1}Ax_j,
\ \tilde B=x_j^{-1}[x_jD_{x_j},A],
\end{equation}
and thus we even
have $\sigma_{\bl,m}(\tilde A)=\sigma_{\bl,m}(A)$.

Indeed, recall from \cite[Lemma~2.5]{Vasy:Propagation-Wave}
that $\Diff^k\Psib^s(X)$ is a filtered algebra with respect to operator
composition, with $B_j\in\Diff^{k_j}\Psib^{s_j}(X)$,
$j=1,2$, implying $B_1 B_2\in\Diff^{k_1+k_2}\Psib^{s_1+s_2}(X)$. Moreover,
with $B_1,B_2$ as above,
\begin{equation*}
[B_1,B_2]\in\Diff^{k_1+k_2}\Psib^{s_1+s_2-1}(X).
\end{equation*}
We also recall the following lemma that computes the principal symbol of
a commutator:

\begin{lemma}(\cite[Lemma~2.8]{Vasy:Propagation-Wave})\label{lemma:comm-symbol}
Let $\pa_{x_j}$, $\pa_{\sigma_j}$ denote local coordinate vector fields
on $\Tb^*X$ in the coordinates $(x,y,\sigma,\zeta)$.
For $A\in\Psib^m(X)$ with Schwartz kernel supported in the coordinate patch,
$a=\sigma_{\bl,m}(A)\in\Cinf(\Tb^* X\setminus o)$,
we have $[D_{x_j},A]=A_1D_{x_j}+A_0\in\Diff^1\Psib^{m-1}(X)$
with $A_0\in\Psib^m(X)$, $A_1\in\Psib^{m-1}(X)$ and
\begin{equation}\label{eq:sigma-A_j}
\sigma_{\bl,m-1}(A_1)
=\frac{1}{\imath}\pa_{\sigma_j}a,\ \sigma_{\bl,m}(A_0)=\frac{1}{\imath}\pa_{x_j}a.
\end{equation}
This result also holds with $\Psib(X)$ replaced by $\Psibc(X)$ everywhere.
\end{lemma}

These results extend immediately to operators acting on sections of
a vector bundle $E$, provided that in the case of \eqref{eq:sigma-A_j}
we require that $A$ has scalar principal symbol, and provided that
we replace $D_{x_j}$ by any $Q\in\Diff^1(X;E)$ with scalar principal
symbol $\xi_j\,\Id_{\Lambda X}$.

Adjoints play a major role in positive commutator estimates, with a
prominent role due to the boundary conditions. We
consider operators acting on functions first before turning to
operators acting on forms. For
the Dirichlet problem the boundary condition
can be handled in a number of ways (essentially
because of the density of $\dCI(X)$ in $H^1_0(X)$), but for other boundary
conditions more care is required.
In \cite{Vasy:Propagation-Wave}, for the Neumann problem, pairings were
considered, and one factor of a differential operator was always left
on each slot of the pairing. Here we use the approach of
\cite{Melrose-Vasy-Wunsch:Corners} to enlarge $\Diff\Psib(X)$ somewhat
by adjoints of differential operators, so that one need not write down
quadratic forms at every point. However, this is mostly only a stylistic
issue.

We first recall the basic function spaces. For $k\geq 0$ integer, we let
$H^k(X)$ be the completion of $\Cinf_\compl(X)$ with respect to
the $H^k(X)$ norm. Then we define $H^k_0(X)$ as the
closure of $\dCinf_\compl(X)$ inside $H^k(X)$.
If $\tilde X$ is a manifold without boundary, and $X$ is embedded into it,
one can also extend elements of $H^k(X)$ to elements $H^k_{\loc}(\tilde X)$.
With H\"ormander's notation
\cite[Appendix~B.2]{Hor}, $H^k_{\loc}(X)=\bar H^k_{\loc}(X^\circ)$
-- \cite[Appendix~B.2]{Hor} discusses the case of a smooth boundary only,
but the general case is similar, see \cite[Section~3]{Vasy:Propagation-Wave}.
As is clear from
the completion definition, $H^k_{0,\loc}(X)$ can be identified with the subset
of $H^k_{\loc}(\tilde X)$ consisting of functions supported in $X$.
Thus, $H^k_{0,\loc}(X)=\dot H^k_{\loc}(X)$ with the notation of
\cite[Appendix~B.2]{Hor}.

We let $H^{-k}(X)$ be the dual of $H^k_0(X)$ and $\dot H^{-k}(X)$ be
the dual of $H^k(X)$, with respect to an extension of the sesquilinear form
$\langle u,v\rangle=
\int_X u\,\overline v\,d\tilde g$, i.e.\ the $L^2$ inner product.
As $H^k_0(X)$ is a closed subspace of $H^k(X)$, $H^{-k}(X)$ is the quotient
of $\dot H^{-k}(X)$ by the annihilator of $H^k_0(X)$, hence there is
a canonical map
$$
\rho:\dot H^{-k}(X)\to H^{-k}(X).
$$
In terms of the identification of the $H^k$
spaces above, $H^{-k}_{\loc}(X)
=\bar H^{-k}_{\loc}(X^\circ)$ in the notation of \cite[Appendix~B.2]{Hor},
i.e.\ its elements are the restrictions to $X^\circ$ of elements
of $H^{-k}_{\loc}(\tilde X)$. Analogously, $\dot H^{-k}_{\loc}(X)$ consists
of those elements of $H^{-k}_{\loc}(\tilde X)$ which are supported in $X$.

If $P\in \Diff^k(X)$, then it defines a continuous linear
map
$$
P:H^k(X)\to L^2(X).
$$
Thus, its Banach space adjoint
(with respect to the {\em sesquilinear} dual pairing) is
a map
\begin{equation}\begin{split}\label{eq:Banach-adjoint}
&P^*:(L^2(X))^*=L^2(X)\to (H^k(X))^*
=\dot H^{-k}(X),\\
&\qquad\qquad
\langle P^* u,v\rangle=\langle u,Pv\rangle,\ u\in L^2(X),
\ v\in H^k(X).
\end{split}\end{equation}

There is an important distinction here between considering $P^*$ as stated,
or as composed with the quotient map, $\rho\circ P^*$.

\begin{lemma}(cf.\ \cite[Lemma~5.18]{Melrose-Vasy-Wunsch:Corners})
\label{lemma:Diff-adjoints}
Suppose that $P\in \Diff^k(X)$. Then there exists a unique
$Q\in \Diff^{k}(X)$ such that $\rho\circ P^*=Q$.
However, in general, acting on $\CI(X)$, $P^*\neq Q$.

If, on the other hand, $P\in \Diffb^{k}(X)$, then there exists a unique
$Q\in \Diffb^{k}(X)$ such that $P^*=Q$.
\end{lemma}

As the proof is a simple modification of that of
\cite[Lemma~5.18]{Melrose-Vasy-Wunsch:Corners} (which actually
deals with a somewhat more complicated case), we omit it here. Indeed,
one can simply regard our setting as a special case of that of
\cite[Lemma~5.18]{Melrose-Vasy-Wunsch:Corners}, namely when one is working away
from the edge that is blown up there -- see
\cite[Section~2]{Melrose-Vasy-Wunsch:Corners} for a discussion of the
relationship.

Following \cite{Melrose-Vasy-Wunsch:Corners}
we now define an extension of $\Diff(X)$ as follows.

\begin{Def}
Let $\Diffd^{k}(X)$ denote the set of Banach space adjoints of elements
of $\Diff^{k}(X)$ in the sense of \eqref{eq:Banach-adjoint}.

Also let $\Diffs^{2k}(X)$ denote
operators of the form
$$
\sum_{j=1}^N Q_j P_j,\ P_j\in \Diff^{k}(X),\ Q_j\in\Diffd^{k}(X).
$$
For $X$ non-compact, the sum is taken to be locally finite.
\end{Def}

Thus, if $P\in \Diffs^{2k}(X)$, $P_j$, $Q_j$ as above, and $Q_j
=R_j^*$, $R_j\in \Diff^{k}(X)$, then
$$
\langle Pu,v\rangle=\sum_{j=1}^N\langle P_j u,R_j v\rangle.
$$

\begin{rem}
While for $R\in\Diff^k(X)$, $R^*\in\Diffd^k(X)$ depends on the inner
product on $L^2(X)$, i.e.\ on the $\CI$ density inducing it, the set
of adjoints is independent of the choice of inner product.
\end{rem}

We now turn to differential operators acting on vector bundles. For this
purpose it is often useful (but is sometimes not absolutely necessary)
to have a positive definite inner product on
$\Lambda X$ (unlike the pairing induced by the Lorentz metric $h$). We will
only consider such inner products
induced by a Riemannian metric $\tilde h$.
Let $H$, resp.\ $\tilde H$
denoting the dual metrics, as well as the induced metrics on forms;
these can be thought of as maps $\Lambda_p X\to (\Lambda_p X)^*$,
hence
\begin{equation}\label{eq:J-def}
J=H^{-1}\tilde H
\end{equation}
is an isomorphism of $\Lambda_p X$.
The inner products then satisfy
\begin{equation}\label{eq:h-ht-ptwise}
(u,Jv)_H=(\tilde H v)(u)=(u,v)_{\tilde H}=(Ju,v)_H,
\end{equation}
and the inner product $(.,.)_{\tilde H}$ is positive definite. In particular
$$
J^*_H=J,\ (J^{-1})^*_{\tilde H}=J^{-1},
$$
where $(.)^*_H$ denotes the adjoint of an endomorphism with respect to
the $H$ inner product.
Thus, with the last equation being a definition,
\begin{equation}\label{eq:h-ht}
\langle u,Jv\rangle_H=\int (u,Jv)_H \,|dh|
=\int (u,v)_{\tilde H} \,|dh|=\langle u,v\rangle_{L^2(X;\Lambda X;|dh|\otimes
\tilde H)}\equiv\langle u,v\rangle,
\end{equation}
where the inner product on the right is thus with respect to $\tilde H$ on
the fibers, but using the density $|dh|$, and is positive definite.

As we only consider natural boundary conditions, and indeed relative
boundary conditions only, we
discuss adjoints only in this setting to avoid overburdening the notation.
Namely, proceeding as above,
if $P\in \Diff^1(X;\Lambda X)$, then it defines a continuous linear
map
$$
P:H^1_R(X;\Lambda X)\to L^2(X;\Lambda X).
$$
Thus, its Banach space adjoint is the map
\begin{equation}\begin{split}\label{eq:Banach-R-adjoint}
&P^*:(L^2(X;\Lambda X))^*=L^2(X;\Lambda X)\to (H^1_R(X;\Lambda X))^*
\equiv\dot H^{-1}_R(X;\Lambda X),\\
&\qquad\qquad
\langle P^* u,v\rangle=\langle u,Pv\rangle,\ u\in L^2(X;\Lambda X),
\ v\in H^1_R(X;\Lambda X).
\end{split}\end{equation}

We extend the preceeding definitions:

\begin{Def}\label{eq:Def-Diffs}
Let $\Diffd^{1}(X;\Lambda X)$ denote the set of
Banach space adjoints of elements
of $\Diff^{1}(X;\Lambda X)$ in the sense of \eqref{eq:Banach-R-adjoint}.

Also let $\Diffs^{2}(X;\Lambda X)$ denote
operators of the form
$$
\sum_{j=1}^N Q_j P_j,\ P_j\in \Diff^{1}(X;\Lambda X),
\ Q_j\in\Diffd^{1}(X;\Lambda X).
$$
For $X$ compact, the sum is taken to be finite; for
$X$ non-compact, it is taken to be locally finite.
\end{Def}

Again, $\Diffs^{2k}(X;\Lambda X)$ is independent of the choice of the
metrics $h$ and $\tilde h$.

However, for calculations below
we need to be somewhat careful in our choice of $\tilde h$. One convenient
choice is a Riemannian metric $\tilde H$ which satisfies
\begin{equation}\label{eq:tilde-H-def}
\tilde H=-H+2\pa_{y_{n-k}}^2
\end{equation}
in a neighborhood of $p_0\in F$ with local coordinates nearby as in
\eqref{eq:B-orth-0}.
Note that $-H+2\pa_{y_{n-k}}^2$
is indeed Riemannian in a sufficiently small neighborhood of the point $p_0$
in view of \eqref{eq:B-orth-0}, so the desired $\tilde H$ exists.
Notice that on this neighborhood $\tilde H(dx_j,.)=H(dx_j,.)$,
thus relative boundary conditions are preserved by $J$. (Note that the
definition of $J$ given in \cite{Vasy:Diffraction-edges} after Equation~(10)
is not correct
in all situations; the present definition should be used instead.)

We can now describe the form of
\begin{equation}\label{eq:Box-as-map}
\Box\in\Diffs^2(X;\Lambda X):H^1_{R,\loc}(X;\Lambda X)
\to (H^1_{R,\compl}(X;\Lambda X))^*.
\end{equation}

\begin{lemma}\label{lemma:Box-form}
Let $\cU$ be a coordinate chart with coordinates $(x_1,\ldots,x_k,
y_1,\ldots,y_{n-k})$ such that \eqref{eq:H-at-corner2} holds,
\eqref{eq:tilde-H-def} is valid, and
trivialize $\Lambda X$ using the coordinate differentials.
With $Q_i=D_{x_i}\otimes\Id_{\Lambda X}$,
the wave operator with relative boundary conditions, i.e.\ as a map
\eqref{eq:Box-as-map},
satisfies
\begin{equation}\label{eq:Box-form}
\Box=\sum_{i,j}Q_i^* A_{ij}(x,y) Q_j+\sum_i (M_i Q_i
+Q_i^* M'_i)+\Ptil\ \text{on } \cU
\end{equation}
with
\begin{equation}\begin{split}
&M_i,M'_i\in \Diffb^{1}(X;\Lambda X),\ \Ptil\in\Diffb^{2}(X;\Lambda X)\\
&\sigma_{\bl,1}(M_i)=m_i\,\Id=\sigma_{\bl,1}(M'_i),
\ \sigma_{\bl,2}(\Ptil)=\ptil\,\Id,\\
& m_i|_F=0,\ \ptil
=\sum_{i,j=1}^{n-k}B_{ij}(x,y)\zeta_i\zeta_j.
\end{split}\end{equation}
\end{lemma}

\begin{proof}
As shown in \cite[Section~5]{Vasy:Diffraction-edges} (which in turn
follows \cite[Section~4]{Mitrea-Taylor-Vasy:Lipschitz}), using
the trivialization to define $\nabla$,
for $u\in\CI_R(X;\Lambda X)$,
$v\in\CI_{R,\compl}(X;\Lambda X)$:
\begin{equation*}
\langle du,dv\rangle_H+\langle\delta u,\delta v\rangle_H
=\langle \nabla u,\nabla v\rangle_H+\langle Ru,v\rangle_H
+\int_{\pa X}(\tilde R u,v)_H\,dS_{\pa X}
\end{equation*}
for some smooth bundle endomorphism $\tilde R$ and a first order
differential operator $R$; by continuity and density
this holds whenever $u\in H^1_{R,\loc}(X;\Lambda X)$,
$v\in H^1_{R,\compl}(X;\Lambda X)$. Thus, the wave equation $\Box u=f$ becomes
\begin{equation}\label{eq:d-del-nabla}
\langle f,v\rangle_H
=\langle \nabla u,\nabla v\rangle_H+\langle Ru,v\rangle_H
+\int_{\pa X}(\tilde R u,v)_H\,dS_{\pa X}
\end{equation}
for all $v\in H^1_{R,\compl}(X;\Lambda X)$.
Rewriting this, replacing $v$ by $Jv\in H^1_{R,\compl}(X;\Lambda X)$ (where
we use that $J$ preserves the boundary condition),
\begin{equation}\begin{split}
\langle f,v\rangle=\langle f,Jv\rangle_H
&=\langle \nabla u,\nabla Jv\rangle_H+\langle Ru,Jv\rangle_H
+\int_{\pa X}(\tilde R u,Jv)_H\,dS_{\pa X}\\
&=\langle \nabla u,\nabla Jv\rangle_H+\langle Ru,v\rangle
+\int_{\pa X}(\tilde R u,v)_{\tilde H}\,dS_{\pa X}\\
\end{split}\end{equation}
for all $v\in H^1_{R,\compl}(X;\Lambda X)$.

Now $R\in\Diff^1(X;\Lambda X)$ means that it can be absorbed in the $M_i Q_i$
and $\tilde P$ terms, as subprincipal terms (i.e.\ the contribution to
the principal symbol of $M_i$ and $\tilde P$ vanishes). Similarly,
$\tilde R$ can be rewritten as the boundary term arising by taking
the adjoint of a first order differential operator, i.e.\ is of the
form $(R')^*-R''$ for $R',R''\in\Diff^1(X;\Lambda X)$, hence again
can be absorbed in the subprincipal terms of the $M_i Q_i$, $Q_i^* M_i'$ and
$\tilde P$ terms.
Thus, it suffices to check the form of
$\langle \nabla u,\nabla Jv\rangle_{L^2(X;\Lambda X)}$.

But writing the coordinates $(x,y)$ as $w$ and using \eqref{eq:h-ht-ptwise}
pointwise,
\begin{equation*}\begin{split}
&\langle\nabla u,\nabla Jv\rangle_H
=\sum_{ij}\int (H_{ij}D_{w_i}u,D_{w_j} Jv)_H\,|dh|\\
&=\sum_{ij}\int (J^{-1}H_{ij}D_{w_i}u,D_{w_j} Jv)_{\tilde H}\,|dh|\\
&=\sum_{ij}\int (H_{ij}D_{w_i}u,J^{-1}D_{w_j} Jv)_{\tilde H}\,|dh|
=\sum_{ij}\langle (J^{-1}D_{w_j}J)^*H_{ij}D_{w_i}u,v\rangle
\end{split}\end{equation*}
where the remaining pairing is the $L^2$-pairing on functions, i.e.\ is
the integral of the product (with a complex conjugation).
Now, $J^{-1}D_{w_j}J=D_{w_j}+T_j$, $T_j\in\CI(X;\End(\Lambda X))$.
Rewriting $w$ as $(x,y)$, and using that the crossterm $C_{ij}(x,y)$
vanishes at $x=0$, we obtain \eqref{eq:Box-form} by noting that the
terms with $T_i$ and $T_j$ can be incorporated in the subprincipal
terms of the $M_i$, $M'_i$ and $\tilde P$ terms of \eqref{eq:Box-form}.
\end{proof}

The operators whose solutions we consider below are perturbations of
$\Box$ by first order operators, i.e.\ we assume that
\begin{equation}\begin{split}\label{eq:P-form}
&P=\Box+P_1:H^1_{R,\loc}(X;\Lambda X)\to \dot H^{-1}_{R,\loc}(X;\Lambda X),\\
&P_1\in\Diff^1(X;\Lambda X)+\Diffd^1(X;\Lambda X).
\end{split}\end{equation}

We now consider adjoints; these are the main cause of difficulty
for commutator constructions. For simplicity (as this is what we need below),
we assume that $\Lambda X$ is trivialized as above, and
$A\in\Psib^m(X;\Lambda X)$ is scalar with respect to this trivialization,
i.e.\ is of the form $A_0\otimes \Id$, $A_0\in\Psib^m(X)$,
$\sigma_{\bl,m}(A_0)=a$. For
$u=(u_\alpha)$, $v=(v_\beta)$ with respect to this trivialization, the fiber
inner
product takes the form
$$
(u,v)=\sum_{\alpha\beta} \tilde H_{\alpha\beta} u_\alpha \overline{v_\beta},
$$
hence the inner product on sections of $\Lambda X$ supported in the
coordinate chart takes the form
$$
\langle u,v\rangle_{L^2(X,\Lambda X;|dh|\otimes\tilde H)}=\int_X
\sum_{\alpha\beta} \tilde H_{\alpha\beta} u_\alpha \overline{v_\beta}\,|dh|
=\sum_{\alpha,\beta}\langle \tilde H_{\alpha\beta}u_{\alpha}
,v_{\beta}\rangle_{L^2(X)}.
$$
In particular, for $A_0$ formally self-adjoint with respect to the
$L^2(X)$-inner product,
\begin{equation*}\begin{split}
\langle u,Av\rangle_{L^2(X,\Lambda X;|dh|\otimes\tilde H)}
&=\sum_{\alpha,\beta}\langle \tilde H_{\alpha\beta}u_{\alpha}
,A_0 v_{\beta}\rangle_{L^2(X)}
=\sum_{\alpha,\beta}\langle A_0 \tilde H_{\alpha\beta}u_{\alpha}
,v_{\beta}\rangle_{L^2(X)}\\
&=\sum_{\alpha,\beta}\langle \tilde H_{\alpha\beta}A_0 u_{\alpha}
,v_{\beta}\rangle_{L^2(X)}+\langle [A_0,\tilde H_{\alpha\beta}] u_{\alpha},
v_{\beta}\rangle_{L^2(X)}\\
&=\langle (A+C)u,v\rangle_{L^2(X,\Lambda X;|dh|\otimes\tilde H)},
\end{split}\end{equation*}
where $C=(C_{\alpha\nu})$ is the matrix form of $C\in\Psib^{m-1}(X;\Lambda X)$
with respect to the trivialization, and
\begin{equation*}\begin{split}
&C_{\alpha\nu}=\sum_{\mu}\tilde h_{\mu\nu}[A_0,\tilde H_{\alpha\mu}],\\
&\qquad
\sigma_{\bl,m-1}(C_{\alpha\nu})=\imath
\sum_{\mu}\tilde h_{\mu\nu}\sH_{\bl,\tilde H_{\alpha\mu}}a,
\end{split}\end{equation*}
In particular, notice that $\sH_{\bl,\tilde H_{\alpha\mu}}$ is a vertical
vector field on the vector bundle
$\Tb^*X$, and at $F=\{x_1=\ldots=x_k=0\}$, it is
a linear combination of the vector fields $\pa_{\zeta_j}$:
\begin{equation}\label{eq:Ham-vf-form-metric}
\sH_{\bl,\tilde H_{\alpha\mu}}=-\sum_j (\pa_{y_j}\tilde H_{\alpha\mu})\pa_{\zeta_j}
-\sum_j x_j(\pa_{x_j}\tilde H_{\alpha\mu})\pa_{\sigma_j}.
\end{equation}
Note that even the principal symbol $C$ does {\em not}
usually preserve boundary conditions, which means that we cannot
estimate an expression like $\langle Cu,\Box u\rangle$ by using the
PDE.

To summarize, we have proved
the following result.

\begin{prop}\label{prop:adjoints}
With $\Lambda X$ trivialized as above, suppose that
$A\in\Psib^m(X;\Lambda X)$ is scalar,
i.e.\ is of the form $A_0\otimes \Id$, $A_0\in\Psib^m(X)$, $A_0$ is
self-adjoint with respect to $|dh|$, $\sigma_{\bl,m}(A_0)=a$. Then
there are smooth vector fields $V_{\alpha\beta}$ on $\Tb^*X$ such that
\begin{equation}\begin{split}
&A^*-A=C,\ \sigma_{\bl,m-1}(C)_{\alpha\beta}=(V_{\alpha\beta} a),\\
&\qquad (\pi_{\Tb^*X\to X})_*(V_{\alpha\beta})=0,
\ V_{\alpha\beta}|_{x_j=0}\sigma_j=0,
\end{split}\end{equation}
where $\pi_{\Tb^*X\to X}:\Tb^*X\to X$ is the bundle projection.
\end{prop}

\begin{rem}\label{rem:subpr-adj-dep}
It is important that $A$ was not merely principally scalar. Indeed,
if we replace $A$ by $A+A_1$, $A_1\in\Psib^{m-1}(X;\Lambda X)$, then
$A^*-A$ is replaced by $C_1=(A^*-A)+(A_1^*-A_1)$, with $A_1^*-A_1\in
\Psib^{m-1}(X;\Lambda X)$, so the principal symbol of $C_1$ is not
determined by $a$.
\end{rem}

In order to have the $A_0$ self-adjoint as above, it is convenient
to introduce the following notation. For $A_0\in\Psib(X)$,
let $A_0^\dagger$
denote the $L^2(X)$-adjoint of $A_0$, and let
\begin{equation}\label{eq:A-dagger-def}
A^\dagger=A_0^\dagger\otimes\Id.
\end{equation}
Thus,
$A^\dagger$ is the $L^2(X;\Lambda X)$ adjoint of $A$ if we put
the Euclidean inner product on the fibers of $\Lambda X$ using the
trivialization, and use $|dh|$ as the density to integrate with
respect to.

First, however, we need to discuss the action of $A\in\Psibc^s(X;\Lambda X)$,
$s\in\RR$,
on $\CI_R(X;\Lambda X)$ and $H^1_R(X;\Lambda X)$. Again, for simplicity
assume that $A$ is supported in a coordinate chart $\cU$ as above. Then $A$
has a normal family
$$
\hat N_{\hsf_j}(A)(\sigma_j):\CI(\hsf_j;\Lambda_{\hsf_j}X)
\to \CI(\hsf_j;\Lambda_{\hsf_j}X),\ \sigma_j\in\RR,
$$
at each boundary hypersurface
$\hsf_j$, $j=1,\ldots,k$, of $X$ intersecting $\cU$, defined by
\begin{equation*}
\hat N_{\hsf_j}(A)(\sigma_j)f=(x_j^{-i\sigma_j}Ax_j^{i\sigma_j}u)|_{\hsf_j},
\ u|_{\hsf_j}=f,
\end{equation*}
where $x_j^{-i\sigma_j}Ax_j^{i\sigma_j}\in\Psibc^s(X)$, hence
$x_j^{-i\sigma_j}Ax_j^{i\sigma_j}u\in\Cinf(X;\Lambda X)$, and the
right hand side does not depend on the choice of $u$. This captures
the behavior of $A$ at $\hsf_j$ in that
$$
(\forall\sigma_j\in\RR)\ \hat N_{\hsf_j}(A)(\sigma_j)=0\Rightarrow
A\in x_j\Psibc^s(X;\Lambda X).
$$
We refer to \cite{Melrose:Atiyah} and \cite{Melrose-Piazza:Analytic}
for more details.

In general, $A\in\Psibc^s(X;\Lambda X)$ does not preserve
$\CI_R(X;\Lambda X)$; for this to happen we need that
for all $j$ and all $\sigma_j$,
\begin{equation}\label{eq:preserve-normal-forms}
\hat N_{\hsf_j}(A)(\sigma_j):
\CI(\hsf_j;\Lambda_{\hsf_j,N}X)
\to \CI(\hsf_j;\Lambda_{\hsf_j,N}X),
\end{equation}
where $\Lambda_{\hsf_j,N}X$ denotes the bundle of normal forms at $\hsf_j$.
If $s\geq 0$, then
in addition $A\in\cL(H^1(X;\Lambda X))$, and thus (recalling that
$\CI_R(X;\Lambda X)$ is dense in $H^1_R(X;\Lambda X)$ in the
$H^1$-norm) $A\in\cL(H^1_R(X;\Lambda X))$.
In particular,
if $A$ is scalar with respect to the coordinate trivialization, then
$\hat N_{\hsf_j}(A)$ satisfies \eqref{eq:preserve-normal-forms}
for each $j$, and it follows that
$A:\CI_R(X;\Lambda X)\to\CI_R(X;\Lambda X)$, and the
corresponding mapping property on $H^1_R(X;\Lambda X)$ also holds for
$s\geq 0$.

We are now ready to state our main commutator result. Recall that the topology
on $\Psibc^s(X;\Lambda X)$ is given by conormal (Besov) seminorms on the
Schwartz kernel, or
equivalently by symbol seminorms (which capture the near diagonal
behavior) combined with $\CI$ seminorms.

\begin{prop}\label{prop:gend-comm}
Let $A_0\in\Psib^0(X)$ with $\sigma_{\bl,0}(A_0)=a$, supported
in the coordinate chart as in Lemma~\ref{lemma:Box-form}, and let
$A=A_0\otimes\Id$ with respect to the trivialization.
Also let $s\in\RR$,
$\Lambda_r$ be scalar, with symbol
\begin{equation}
w_r=|\zeta_{n-k}|^{s+1/2}(1+r|\zeta_{n-k}|^2)^{-s}\,\Id,\quad r\in[0,1),
\end{equation}
so $A_r=A\Lambda_r\in\Psib^{0}(X;\Lambda X)$ for $r>0$ and the family
$\cA=\{A_r:\ r\in(0,1)\}$ is
uniformly bounded in $\Psibc^{s+1/2}(X;\Lambda X)$. Then, for
$P$ as in \eqref{eq:P-form}, and with $A_r^\dagger$
as in \eqref{eq:A-dagger-def}, we have
\begin{equation}\begin{split}\label{eq:gend-commutator}
&\imath(A^\dagger_r A_r)^* P-\imath P A^\dagger_r A_r =
Q_i^* C_{r,ij}Q_j+Q_i^* C'_{r,i}+C''_{r,j}Q_j+C_{r,0}+F_r,\\
&\sigma_{\bl,2s}(C_{r,ij})=2w_r^2 (a V_{ij} a\,\Id_{\Lambda X}+a A_{ij}
\tilde V a
 + a^2 \tilde c_{r,ij}),\\
&\sigma_{\bl,2s+1}(C'_{r,i})=\sigma_{\bl,2s}(C''_{r,i})
=2w_r^2 (a V_{i} a\,\Id_{\Lambda X}+a m_i\tilde V a
 + a^2 \tilde c_{r,i}),\\
&\sigma_{\bl,2s+2}(C_{r,0})=2w_r^2 (a V_{0} a\,\Id_{\Lambda X}
+a \tilde p\tilde V a
 + a^2 \tilde c_{r,0}),\\
&\tilde c_{r,ij}\in L^\infty
((0,1]_r;S^{-1}(T^*X\setminus o;\End(\Lambda X))\\
&\tilde c_{r,i}\in L^\infty
((0,1]_r;S^{0}(T^*X\setminus o;\End(\Lambda X)),\ \tilde c_{r,0}\in L^\infty
((0,1]_r;S^{1}(T^*X\setminus o;\End(\Lambda X))\\
&V_{ij}\in\Vf(T^*X\setminus o),\ V_{i}\in\Vf(T^*X\setminus o),\\
&V_{0}\in\Vf(T^*X\setminus o),
\ \tilde V\in\Vf(T^*X\setminus o;\End(\Lambda X)),\\
&F_r\in L^\infty((0,1]_r;\Diffs^2\Psib^{2s-1}(X;\Lambda X)),
\end{split}\end{equation}
$\tilde V,V_{ij},V_i, V_0$
smooth homogeneous of degree $-1,-1,0,1$ respectively, $\tilde V$ is
vertical
and annihilates $\sigma_i$ at $x_i=0$.
Moreover,
\begin{equation}\begin{split}\label{eq:V_ij-form}
V_{ij}|_F=-A_{ij}(\pa_{\sigma_i}+\pa_{\sigma_j})
+\sum_k (\pa_{y_k} A_{ij})\pa_{\zeta_k},\\
V_i|_F=-A_{ij}\pa_{x_j},\ V_0|_F=-\sH_{\bl,\tilde p}.
\end{split}\end{equation}
\end{prop}

\begin{proof}
We use \eqref{eq:Box-form} and $[B^*,Q_i^*]=B^*Q_i^*-Q_i^*B^*=[Q_i,B]^*$, etc.
For instance,
\begin{equation}\begin{split}\label{eq:gend-commutator-1a}
&(A^\dagger_r A_r)^* Q_i^* A_{ij}Q_j-Q_i^*A_{ij}Q_j A^\dagger_r A_r \\
&=[Q_i,A^\dagger_rA_r]^*A_{ij}Q_j\\
&\qquad-Q_i^*A_{ij}[Q_j,A^\dagger_r A_r]
+Q_i^* \Big((A^\dagger_r A_r)^*A_{ij}-A_{ij}A^\dagger A \Big)Q_j,
\end{split}\end{equation}
and
\begin{equation*}
(A^\dagger_r A_r)^*A_{ij}-A_{ij}A^\dagger_r A_r=((A^\dagger_r A_r)^*
-(A^\dagger_r A_r))A_{ij}+[A^\dagger_r A_r,A_{ij}].
\end{equation*}
We have already calculated $(A^\dagger_r A_r)^*-(A^\dagger_r A_r)$,
including its principal symbol, in Proposition~\ref{prop:adjoints},
while the principal symbol of
$[A^\dagger_r A_r,A_{ij}]\in\Psibc^{2s}(X)$ can be computed
in $\Psibc^{2s}(X)$. Further, $[Q_j,A^\dagger_r A_r]$ can be computed
using Lemma~\ref{lemma:comm-symbol}. Note that in all these
terms the principal symbol is given by applying a vector field
to $a|\zeta_{n-k}|^{s+1/2}(1+r|\zeta_{n-k}|^2)^{-s}$.
When $|\zeta_{n-k}|^{s+1/2}(1+r|\zeta_{n-k}|^2)^{-s}$
is differentiated, $a$ is not differentiated, hence it contributes to
the term $a^2\tilde q_r$. Thus, we need to collect the terms in which
$a$ is differentiated. Apart from the contribution of
$(A^\dagger_r A_r)^*-(A^\dagger_r A_r)$, these are all principally scalar.
The contribution of $(A^\dagger_r A_r)^*-(A^\dagger_r A_r)$ gives
rise to the $\tilde V$ terms; the verticality and the property
of annihilating $\sigma_i$ follow from
Proposition~\ref{prop:adjoints}. The other terms in which $a$ is
differentiated give $V_{ij},V_i$ and $V_0$. In particular, the
$V_{ij}$ arises from
the $A_1$ term in Lemma~\ref{lemma:comm-symbol} when either
$Q_i^*$ or $Q_j$ in $Q_i^* A_{ij} Q_j$ is commuted with $A^\dagger A$
as well as when $A_{ij}$ is commuted with $A^\dagger A$; these give
rise to the three terms for $V_{ij}|_F$ in \eqref{eq:V_ij-form}
respectively.
The $V_i$ term arises both from the $A_0$ term in
Lemma~\ref{lemma:comm-symbol} when either
$Q_i^*$ or $Q_j$ in $Q_i^* A_{ij} Q_j$ is commuted with $A^\dagger A$
(which gives $V_i|_F$ in \eqref{eq:V_ij-form}),
as well as the $A_1$ term in Lemma~\ref{lemma:comm-symbol} when
$Q_i^*$ in $Q_i^* M_i'$ or $Q_i$ in $M_i Q_i$ is commuted with $A^\dagger A$,
as well as when $M_i$ or $M_i'$ is commuted with $A^\dagger A$: note that
all but the first of these have vanishing principal symbol at $F$
as $m_i|_F=0$.
Finally, the $V_0$ term arises from the $A_0$ term in
Lemma~\ref{lemma:comm-symbol} when
$Q_i^*$ in $Q_i^* M_i'$ or $Q_i$ in $M_i Q_i$ is commuted with $A^\dagger A$,
as well as when $\tilde P$ is commuted with $A^\dagger A$: all but the last
of these have vanishing principal symbol at $F$ as $m_i|_F=0$.
\end{proof}

\section{Elliptic estimates}\label{sec:elliptic}
We collect here the elliptic estimates from \cite{Vasy:Propagation-Wave},
whose validity for forms with natural boundary conditions was
discussed in \cite{Vasy:Diffraction-edges}, though they were
not all stated as explicit lemmas there. Recall that
$$
\langle.,.\rangle
=\langle.,.\rangle_{L^2(X,\Lambda X;|dh|\otimes\tilde H)}
$$
is a positive definite inner product, and $\langle.,.\rangle_H$ is the
metric inner product with respect to $H$.
We also let
\begin{equation}\label{eq:twisted-inner-product}
\langle.,.\rangle_{\tilde H\otimes H}
=\langle.,.\rangle_{L^2(X,\Lambda X\otimes T^*X;|dh|\otimes \tilde H
\otimes H)}
=\int(u,v)_{\tilde H\otimes H}\,|dh|,
\end{equation}
where
$(u,v)_{\tilde H\otimes H}$ is the inner product on
$\Lambda X\otimes T^*X$ with the $T^*X$ inner product given by
$H$ and the $\Lambda X$ inner product given by $\tilde H$ -- cf.\ the
twisted Dirichlet form in \cite[Equation (28)]{Vasy:Diffraction-edges}.

We recall the convention for `local norms' from
\cite[Remark~4.1]{Vasy:Propagation-Wave}:

\begin{rem}\label{rem:localize}
Since $X$ is non-compact and our results are microlocal, we may always
fix a compact set $\tilde K\subset X$ and assume that all ps.d.o's
have Schwartz kernel supported in $\tilde K\times\tilde K$. We
also let $\tilde U$ be a neighborhood of $\tilde K$ in $X$ such that
$\tilde U$ has compact closure, and use the $H^1(\tilde U)$ norm
in place of the $H^1(X)$ norm to accommodate $u\in H^1_{\loc}(X)$.
(We may instead take $\phi\in\Cinf_\compl(\tilde U)$ identically $1$ in a
neighborhood of $\tilde K$, and use $\|\phi u\|_{H^1(X)}$.)
Here we use the notation $\|.\|_{H^1_{\loc}(X;\Lambda X)}$ for
$\|.\|_{H^1(\tilde U;\Lambda X)}$ to avoid having to specify $\tilde U$;
indeed we usually drop $(X;\Lambda X)$ and the subscript $R$ as well.
We also use $\|v\|_{\dot H^{-1}_{R,\loc}(X;\Lambda X)}$ for
$\|\phi v\|_{\dot H^{-1}_R(X;\Lambda X)}$.
\end{rem}

For all the estimates in this section, namely
Lemmas~\ref{lemma:Dirichlet-form}, \ref{lemma:Dirichlet-form-2} and
\ref{lemma:Dt-Dx},
we fix a coordinate chart,
the corresponding trivialization of $\Lambda X$, and let
$\nabla$ be the connection on $\Lambda X$ given by the trivialization, so
\begin{equation*}
\nabla\in\Diff^1(X;\Lambda X;\Lambda X\otimes T^*X),
\end{equation*}
and
\begin{equation}\label{eq:nabla-symbol}
\sigma_1(\nabla)(w,\tilde\xi)=\imath\Id\otimes\tilde\xi\in\End(\Lambda_w X;
\Lambda_w X\otimes T^*_wX),\ (w,\tilde\xi)\in T^*X\setminus o.
\end{equation}

First the basic estimate
on the Dirichlet form is:

\begin{lemma}(cf.\ \cite[Lemma~4.2]{Vasy:Propagation-Wave}; see
\cite[Section~5]{Vasy:Diffraction-edges} for how the proof
of \cite[Lemma~4.2]{Vasy:Propagation-Wave} needs to be modified in this
case.)
\label{lemma:Dirichlet-form}
Suppose that $P$ is as in \eqref{eq:P-form}.
Suppose that $K\subset\Sb^*X$ is compact, $U\subset\Sb^*_{\cU}X$ is open,
$K\subset U$, $\overline{\cU}\subset\cU_0$.
Suppose that $\cA=\{A_r:\ r\in(0,1]\}$ is a bounded
family of scalar ps.d.o's in $\Psibc^s(X;\Lambda X)$ with
$\WFb'(\cA)\subset K$, and
with $A_r\in\Psib^{s-1}(X;\Lambda X)$ for $r\in (0,1]$.
Then there are $G\in\Psib^{s-1/2}(X;\Lambda X)$,
$\tilde G\in\Psib^{s+1/2}(X;\Lambda X)$ scalar
with $\WFb'(G),\WFb'(\tilde G)\subset U$
and $C_0>0$ such that for $r\in(0,1]$, $u\in H^1_{R,\loc}(X;\Lambda X)$
with $\WFb^{1,s-1/2}(u)
\cap U=\emptyset$, $\WFb^{-1,s+1/2}(Pu)\cap U=\emptyset$, we have
\begin{equation}\begin{split}\label{eq:Dirichlet-form}
&|\langle \nabla A_r u, \nabla A_r u\rangle_{\tilde H\otimes H}|\\
&\qquad\leq
C_0\big(\|u\|^2_{H^1_{\loc}}+\|Gu\|^2_{H^1}+\|Pu\|^2_{\dot H^{-1}_{R,\loc}}
+\|\tilde G Pu\|^2_{\dot H^{-1}_R}\big).
\end{split}\end{equation}
\end{lemma}

\begin{rem}
It is straightforward to modify this lemma so that
we do not need to assume $U\subset\Sb^*_{\cU}X$ is open,
$K\subset U$, $\overline{\cU}\subset\cU_0$, rather simply
$U\subset\Sb^*X$ is open, and also $A_r$ needs to be merely principally
scalar rather than scalar, and $\nabla$ can be replaced by any first order
differential operator with principal symbol \eqref{eq:nabla-symbol}.
Indeed, we merely need to use a partition of unity and
observe that any new terms introduced by the partition of unity and
the other changes
can be absorbed into
$C_0(\|u\|^2_{H^1_{\loc}}+\|Gu\|^2_{H^1})$, by possibly adjusting
$C_0$ and $G$ (but keeping its properties).
However, as the setting relevant to
our estimates is local, and we choose $A_r$ (which we choose to
be scalar), this is not needed here.
\end{rem}

A slightly strengthened version in terms of the order of $\tilde G$
(corresponding to the right hand side of the equation $Pu=f$) is:

\begin{lemma}(cf.\ \cite[Lemma~4.4]{Vasy:Propagation-Wave}; see
\cite[Section~5]{Vasy:Diffraction-edges} for how the proof
of \cite[Lemma~4.2]{Vasy:Propagation-Wave} needs to be modified in this
case.)
\label{lemma:Dirichlet-form-2}
Suppose that $P$ is as in \eqref{eq:P-form}.
Suppose that $K\subset\Sb^*X$ is compact, $U\subset\Sb^*_{\cU}X$ is open,
$K\subset U$, $\overline{\cU}\subset\cU_0$.
Suppose that $\cA=\{A_r:\ r\in(0,1]\}$ is a bounded
family of scalar ps.d.o's in $\Psibc^s(X;\Lambda X)$ with
$\WFb'(\cA)\subset K$, and
with $A_r\in\Psib^{s-1}(X;\Lambda X)$ for $r\in (0,1]$.
Then there are $G\in\Psib^{s-1/2}(X;\Lambda X)$,
$\tilde G\in\Psib^{s}(X;\Lambda X)$ scalar
with $\WFb'(G),\WFb'(\tilde G)\subset U$
and $C_0>0$ such that for $\ep>0$, $r\in(0,1]$, $u\in H^1_{R,\loc}(X;\Lambda X)$
with $\WFb^{1,s-1/2}(u)
\cap U=\emptyset$, $\WFb^{-1,s}(Pu)\cap U=\emptyset$, we have
\begin{equation*}\begin{split}
|\langle \nabla A_r u, \nabla A_r u\rangle_{\tilde H\otimes H}|\leq
\ep&\|(D_{y_{n-k}}\otimes\Id) A_r u\|^2_{L^2}
+C_0\big(\|u\|^2_{H^1_{\loc}}+\|Gu\|^2_{H^1}\\
&\qquad\qquad+\ep^{-1}\|Pu\|^2_{\dot H^{-1}_{R,\loc}}
+\ep^{-1}\|\tilde G Pu\|^2_{\dot H^{-1}_R}\big).
\end{split}\end{equation*}
\end{lemma}

We then recall the statement of microlocal elliptic regularity from
\cite{Vasy:Diffraction-edges}:

\begin{prop}\label{prop:elliptic}(Microlocal elliptic regularity,
see \cite[Theorem~9]{Vasy:Diffraction-edges}.)
Suppose that $P$ is as in \eqref{eq:P-form}, $m\in\RR$ or $m=\infty$.
Suppose $u\in H^1_{R,\loc}(X;\Lambda X)$. Then
\begin{equation*}
\WFb^{1,m}(u)\setminus\dot\Sigma\subset \WFb^{-1,m}(Pu).
\end{equation*}
\end{prop}

We also need a result giving more precise control of $\|Q_i A_r u\|$.
Given Lemma~\ref{lemma:Dirichlet-form}, the proof proceeds exactly
as in \cite[Lemma~7.1]{Vasy:Propagation-Wave}, namely Equation~(7.2)
follows as there, and from that point the argument is a b-ps.d.o.\ argument
(rather than a $\Diffs\Psib$ argument), and is unaffected by the boundary
conditions.

\begin{lemma}(cf.\ \cite[Lemma~7.1]{Vasy:Propagation-Wave})\label{lemma:Dt-Dx}
Suppose that $P$ is as in \eqref{eq:P-form}.
Suppose $u\in H^1_{R,\loc}(X;\Lambda X)$,
and suppose that we are given
$K\subset\Sb^*_{\cU}X$ compact, $\overline{\cU}\subset\cU_0$, satisfying
\begin{equation*}
K\subset\cG\cap \Sb^*_{F_{\reg}}X\setminus\WFb^{-1,s+1/2}(Pu).
\end{equation*}
Then
there exist $\delta_0>0$ and $C_{\cG,K}>0$ with the following property.
Let $\delta<\delta_0$,
$U\subset\Sb^*X$ open in a $\delta$-neighborhood of $K$,
and $\cA=\{A_r:\ r\in(0,1]\}$ be a bounded
family of scalar
ps.d.o's in $\Psibc^s(X;\Lambda X)$ with $\WFb'(\cA)\subset U$, and
with $A_r\in\Psib^{s-1}(X;\Lambda X)$ for $r\in (0,1]$.

Then
there exist
$$
G\in\Psib^{s-1/2}(X;\Lambda X),\ \tilde G\in\Psib^{s+1/2}(X;\Lambda X)
$$
scalar with $\WFb'(G),\WFb'(\tilde G)\subset U$ and $\tilde C_0
=\tilde C_0(\delta)>0$ such that
for all $r>0$,
\begin{equation*}\begin{split}
\sum_i\|Q_i A_r u\|^2
\leq C_{\cG,K}\delta\|(D_{y_{n-k}}\otimes\Id) A_r u\|^2+\tilde C_0\big(&\|u\|^2_{H^1_{\loc}}
+\|Gu\|^2_{H^1}\\
&+\|Pu\|^2_{\dot H^{-1}_{R,\loc}}
+\|\tilde G Pu\|^2_{\dot H^{-1}_R}\big).
\end{split}\end{equation*}
Here $Q_i=D_{x_i}\otimes\Id$ as in Lemma~\ref{lemma:Box-form},
and $D_{y_{n-k}}\otimes\Id$ is defined with respect to the same
trivialization.
\end{lemma}

\section{Normal propagation}\label{sec:normal}

We now turn to propagation of singularities at hyperbolic points.
Recall from \eqref{eq:b-sections}
that $\sigma_j$ is the b-dual variable of $x_j$,
$\hat\sigma_j=\sigma_j/|\zeta_{n-k}|$.

\begin{prop}\label{prop:normal-prop}(Normal propagation.)
Suppose that $P$ is as in \eqref{eq:P-form}, i.e.\ consider $P$
with relative boundary conditions.
Let $q_0=(0,y_0,0,\zeta_0)\in\cH\cap \Tb^*_{F_{\reg}} X$,
$F\cap U=U
\cap\{x=0\}$,
and let
$$
\eta=-\sum_j\sigmah_j
$$
be the
function defined
in the local coordinates discussed above, and suppose that $u\in
H^1_{R,\loc}(X;\Lambda X)$, $q_0\notin\WFb^{-1,\infty}(f)$,
$f=P u$.
If there exists a conic neighborhood $U$ of $q_0$ in $\Tb^*X\setminus o$
such that
\begin{equation}\begin{split}\label{eq:prop-9a}
q\in U\Mand \eta(q)<0\Rightarrow q\notin\WFb^{1,\infty}(u)
\end{split}\end{equation}
then $q_0\notin\WFb^{1,\infty}(u)$.

In fact, if the wave front set assumptions are relaxed to
$q_0\notin\WFb^{-1,s+1}(f)$ ($f=P u$)
and the existence of a conic neighborhood $U$ of
$q_0$ in $\Tb^*X\setminus o$ such that
\begin{equation}\begin{split}\label{eq:prop-9a-s}
q\in U\Mand \eta(q)<0\Rightarrow q\notin\WFb^{1,s}(u),
\end{split}\end{equation}
then we can still conclude that $q_0\notin\WFb^{1,s}(u)$.
\end{prop}

\begin{rem}\label{rem:normal-remark}
The analogous result also holds for absolute boundary condition, either
by a simple modification of the proof given below, or simply using the
Hodge star operator to move between the boundary conditions.

As follows immediately from the proof given below,
in \eqref{eq:prop-9a} and \eqref{eq:prop-9a-s},
one can replace $\eta(q)<0$ by $\eta(q)>0$, i.e.\ one has the conclusion
for either direction (backward or forward) of propagation.

Moreover, every neighborhood $U$ of
$q_0=(y_0,\zeta_0)\in\cH\cap \Tb^*_{F_\reg}X$ in $\dot\Sigma$
contains an open set of the form
\begin{equation}\label{eq:prop-rem-9b}
\{q:\ |x(q)|^2+|y(q)-y_0|^2+|\zetah(q)
-\zetah_0|^2<\delta\},
\end{equation}
see \cite[Equation~(5.1)]{Vasy:Propagation-Wave}.
Note also that
\eqref{eq:prop-9a} implies the same statement with $U$ replaced by
any smaller neighborhood of $q_0$; in particular, for
the set \eqref{eq:prop-rem-9b}, provided that $\delta$ is sufficiently small.
We can also assume by the same observation that $\WFb^{-1,s+1}(Pu)\cap U
=\emptyset$. Furthermore, with $\tilde p=\sigma_{\bl,2}(\tilde P)$, we
can also arrange that $\tilde p(x,y,\sigma,\zeta)>|(\sigma,\zeta)|^2
|\zeta_0|^{-2}
\tilde p(q_0)/2$ on $U$ since
$\zeta_0\cdot B(y_0)\zeta_0=\tilde p(0,y_0,0,\zeta_0)>0$.
\end{rem}

\begin{proof}
We first construct a commutant by defining its scalar principal symbol, $a$.
This completely follows the scalar case, see \cite[Proof of
Proposition~6.2]{Vasy:Propagation-Wave}. Next we show how to obtain
the desired estimate.

So, as in \cite[Proof of
Proposition~6.2]{Vasy:Propagation-Wave}, let
\begin{equation}\label{eq:prop-omega-def}
\omega(q)=|x(q)|^2+|y(q)-y_0|^2+|\zetah(q)-\zetah_0|^2,
\end{equation}
with $|.|$ denoting the Euclidean norm.
For $\ep>0$, $\delta>0$, with other restrictions to be imposed later on,
let
\begin{equation}\label{eq:normal-phi-def}
\phi=\eta+\frac{1}{\ep^2\delta}\omega,
\end{equation}
Let $\chi_0\in\Cinf(\Real)$ be equal to $0$ on $(-\infty,0]$ and
$\chi_0(t)=\exp(-1/t)$ for $t>0$. Thus,
$t^2\chi_0'(t)=\chi_0(t)$ for $t\in\RR$.
Let $\chi_1\in\Cinf(\Real)$ be $0$
on $(-\infty,0]$, $1$ on $[1,\infty)$, with $\chi_1'\geq 0$ satisfying
$\chi_1'\in\Cinf_\compl((0,1))$. Finally, let $\chi_2\in\Cinf_\compl(\Real)$
be supported in $[-2c_1,2c_1]$, identically $1$ on $[-c_1,c_1]$,
where $c_1$ is such that if
$|\sigmah|^2<c_1/2$ in $\dot\Sigma\cap U_0$. Thus,
$\chi_2(|\sigmah|^2)$ is a cutoff in $|\sigmah|$, with its
support properties ensuring that $d\chi_2(|\sigmah|^2)$ is supported in
$|\sigmah|^2\in [c_1,2c_1]$ hence outside $\dot\Sigma$ -- it should
be thought of as a factor that microlocalizes near the characteristic
set but effectively commutes with $P$.
Then, for $\cte>0$ large, to be determined, let
\begin{equation}\label{eq:prop-22}
a=\chi_0(\cte^{-1}(2-\phi/\delta))\chi_1(\eta/
\delta+2)\chi_2(|\sigmah|^2);
\end{equation}
so $a$ is a homogeneous degree zero $\Cinf$ function on a conic neighborhood
of $q_0$ in $\Tb^*X\setminus o$.
Indeed, as we see momentarily, for any $\ep>0$,
$a$ has compact
support inside this neighborhood (regarded as a subset of $\Sb^*X$, i.e.
quotienting out by the $\Real^+$-action) for $\delta$ sufficiently small,
so in fact it is globally well-defined.
In fact, on $\supp a$ we have $\phi\leq 2\delta$ and $\eta\geq-2\delta$.
Since $\omega\geq 0$, the first of these inequalities implies that
$\eta\leq 2\delta$, so on $\supp a$
\begin{equation}
|\eta|\leq 2\delta.
\end{equation}
Hence,
\begin{equation}\label{eq:omega-delta-est}
\omega\leq \ep^2\delta(2\delta-\eta)\leq4\delta^2\ep^2.
\end{equation}
In view of \eqref{eq:prop-omega-def} and \eqref{eq:prop-rem-9b},
this shows that given any $\ep_0>0$ there exists $\delta_0>0$ such that
for any $\ep\in (0,\ep_0)$ and $\delta\in(0,\delta_0)$, $a$ is supported
in $U$.
The role that $\cte$ large
plays (in the definition of $a$)
is that it increases the size of the first derivatives of $a$ relative
to the size of $a$, hence it allows us to give a bound for
$a$ in terms of a small multiple of
its derivative along the Hamilton vector field.

Now let $A_0\in\Psib^0(X)$ with $\sigma_{\bl,0}(A_0)=a$, supported
in the coordinate chart, and let
$A=A_0\otimes\Id$ with respect to the trivialization.
Also let
$\Lambda_r$ be scalar, have symbol
\begin{equation}
|\zeta_{n-k}|^{s+1/2}(1+r|\zeta_{n-k}|^2)^{-s}\,\Id,\quad r\in[0,1),
\end{equation}
so $A_r=A\Lambda_r\in\Psib^{0}(X;\Lambda X)$ for $r>0$ and it is
uniformly bounded in $\Psibc^{s+1/2}(X;\Lambda X)$.
Then, for $r>0$,
\begin{equation}\label{eq:pairing-identity}
\langle \imath Pu,A^\dagger_r A_r u\rangle-\langle \imath A^\dagger_r A_r  u,Pu\rangle
=\langle \imath(A^\dagger_r A_r)^* Pu,u\rangle
-\langle \imath P  A^\dagger_r A_r u,u\rangle.
\end{equation}
We can compute this using Proposition~\ref{prop:gend-comm}. We arrange
the terms of the proposition so that the terms in which a vector
field differentiates $\chi_1$ are included in $E_r$, the terms in which
a vector fields differentiates $\chi_2$ are included in $E'_r$. Thus, we
have
\begin{equation}\label{eq:gend-commutator-normal}
\imath(A^\dagger_r A_r)^* P-\imath P A^\dagger_r A_r =
Q_i^* C_{r,ij}Q_j+Q_i^* C'_{r,i}+C''_{r,j}Q_j+C_{r,0}+E_r+E'_r+F_r,
\end{equation}
with
\begin{equation}\begin{split}\label{eq:gend-commutator-normal-symbol}
&\sigma_{\bl,2s}(C_{r,ij})=w_r^2\Big(4\digamma^{-1}\delta^{-1}a
|\zeta_{n-k}|^{-1}(-A_{ij}+\hat f_{ij}+\ep^{-2}\delta^{-1}
f_{ij})
\chi_0'\chi_1\chi_2+a^2 \tilde c_{r,ij}\Big),\\
&\sigma_{\bl,2s+1}(C'_{r,i})=w_r^2\Big(\digamma^{-1}\delta^{-1}a
(\hat f'_i+\delta^{-1}\ep^{-2}f'_i)
\chi_0'\chi_1\chi_2+a^2 \tilde c'_{r,i}\Big),\\
&\sigma_{\bl,2s+1}(C''_{r,i})=w_r^2\Big(\digamma^{-1}\delta^{-1}a
(\hat f''_i+\delta^{-1}\ep^{-2}f''_i)
\chi_0'\chi_1\chi_2+a^2 \tilde c''_{r,i}\Big),\\
&\sigma_{\bl,2s+2}(C_{r,0})
=w_r^2\Big(\digamma^{-1}\delta^{-1}|\zeta_{n-k}|a
(\hat f_0+\delta^{-1}\ep^{-2}f_0)
\chi_0'\chi_1\chi_2+a^2 \tilde c_{r,0}\Big),
\end{split}\end{equation}
where $f_{ij}$, $f'_i$, $f''_i$ and $f_0$ as well as $\hat f_{ij}, \hat f'_i,
\hat f''_i$ and $\hat f_0$ are all smooth
$\End(\Lambda X)$-valued
functions on $\Tb^* X\setminus o$,
homogeneous of degree 0 (independent of
$\ep$ and $\delta$). Moreover, $f_{ij}, f'_i, f''_i, f_0$ arise from when
$\omega$ is differentiated in $\chi(\cte^{-1}(2-\phi/\delta))$, and thus
vanish when $\omega=0$, while $\hat f_{ij}, \hat f'_i,
\hat f''_i$ and $\hat f_0$ arise when $\eta$ is differentiated in
$\chi(\cte^{-1}(2-\phi/\delta))$, and comprise all such terms with
the exception of those arising from the $\pa_{\sigma_i}$ and $\pa_{\sigma_j}$
components of $V_{ij}|_F$ (which give $A_{ij}$ on the first line above)
hence are the sums of functions vanishing at $x=0$ (corresponding to
us only specifying the restrictions of the vector fields in
\eqref{eq:V_ij-form} at $F$) and functions vanishing at $\hat\sigma=0$
(when $|\zeta_{n-k}|^{-1}$ in $\eta=-\sum_j\sigma_j|\zeta_{n-k}|^{-1}$
is differentiated)\footnote{Terms of the latter kind
did not occur in \cite{Vasy:Propagation-Wave}
as time-translation invariance was assumed, but it does occur in
\cite{Vasy:Diffraction-edges}, where the Lorentzian scalar setting
is considered.}.

In this formula we think of
\begin{equation}\label{eq:normal-main-term}
-4\digamma^{-1}\delta^{-1}w_r^2 a
|\zeta_{n-k}|^{-1}A_{ij}\chi_0'\chi_1\chi_2
\end{equation}
as the main term; note that $-A_{ij}$ is positive definite.
Compared to this,
the terms with $a^2$ are negligible, for they can all
be bounded by
$$
c\digamma^{-1}(\digamma^{-1}\delta^{-1}w_r^2 a
|\zeta_{n-k}|^{-1}\chi_0'\chi_1\chi_2)
$$
(cf.\ \eqref{eq:normal-main-term}), i.e.\ by a small multiple of
$\digamma^{-1}\delta^{-1}w_r^2 a
|\zeta_{n-k}|^{-1}\chi_0'\chi_1\chi_2$
when $\digamma$ is taken large, using that
$2-\phi/\delta\leq 4$ on $\supp a$ and
\begin{equation}\label{eq:chi_0-chi_0-prime}
\chi_0(\digamma^{-1}t)=(\digamma^{-1}t)^2
\chi_0'(\digamma^{-1}t)\leq 16\digamma^{-2}\chi'_0(\digamma^{-1}t),
\ t\leq 4;
\end{equation}
see the discussion in \cite[Section~6]{Vasy:Diffraction-edges}
and \cite{Vasy:Propagation-Wave} following
Equation~(6.19).

The vanishing condition on the $f_{ij}$ and $f_i$ ensures that,
with $|.|$ denoting norms in $\End(\Lambda X)$, on $\supp a$,
\begin{equation}\label{eq:normal-f-estimates}
|f_{ij}|,|f'_i|,|f''_i|,|f_0|\leq C\omega^{1/2}\leq 2C\ep\delta,
\end{equation}
so the corresponding terms can thus
be estimated using $w_r^2\digamma^{-1}\delta^{-1}a
|\zeta_{n-k}|^{-1}\chi_0'\chi_1\chi_2$ provided $\ep^{-1}$ is not too
large, i.e.\ there exists $\tilde\ep_0>0$ such that if $\ep>\tilde\ep_0$,
the terms with $f_{ij}$ can be treated as error terms.

On the other hand, we have
\begin{equation}\label{eq:normal-hat-f-estimates}
|\hat f_{ij}|,|\hat f'_i|,|\hat f''_i|,|\hat f_0|
\leq C|x|+C|\hat\sigma|\leq C\omega^{1/2}+C|\hat\sigma|
\leq 2C\ep\delta+C|\hat\sigma|.
\end{equation}
Now, on $\dot\Sigma$, $|\hat\sigma|\leq 2|x|$ (for $|\sigma_j|=|x_j||\xi_j|\leq
2|x_j| |\zeta_{n-k}|$ with $U$ sufficiently small). Thus we can write
$\hat f_{ij}=\hat f_{ij}^\sharp+\hat f_{ij}^\flat$ with
$\hat f_{ij}^\flat$ supported away from $\dot\Sigma$ and
$\hat f_{ij}^\sharp$ satisfying
\begin{equation}\label{eq:normal-hat-f-sharp-estimates}
|\hat f_{ij}^\sharp|\leq C|x|+C|\hat\sigma|\leq C'|x|\leq
C'\omega^{1/2}\leq 2C'\ep\delta;
\end{equation}
we can also obtain a similar decomposition for $\hat f'_i,\hat f''_i,
\hat f_0$.

Indeed, using \eqref{eq:chi_0-chi_0-prime} it is useful to rewrite
\eqref{eq:gend-commutator-normal-symbol} as
\begin{equation}\begin{split}\label{eq:gend-commutator-normal-symbol-2}
&\sigma_{\bl,2s}(C_{r,ij})=4w_r^2\digamma^{-1}\delta^{-1}a
|\zeta_{n-k}|^{-1}(-A_{ij}+\hat f_{ij}+\ep^{-2}\delta^{-1}
f_{ij}+\digamma^{-1}\delta\hat c_{r,ij})
\chi_0'\chi_1\chi_2,\\
&\sigma_{\bl,2s+1}(C'_{r,i})=w_r^2\delta^{-1}\digamma^{-1}a(\hat f'_i+
\delta^{-1}\ep^{-2}f'_i+
\digamma^{-1}\delta\hat c'_{r,i})
\chi_0'\chi_1\chi_2,\\
&\sigma_{\bl,2s+1}(C''_{r,i})=w_r^2\delta^{-1}\digamma^{-1}a(\hat f''_i+
\delta^{-1}\ep^{-2}f''_i+
\digamma^{-1}\delta\hat c''_{r,i})
\chi_0'\chi_1\chi_2,\\
&\sigma_{\bl,2s+2}(C_{r,0})
=w_r^2\delta^{-1}\digamma^{-1}a|\zeta_{n-k}|
(\hat f_0+\delta^{-1}\ep^{-2}f_0+\digamma^{-1}\hat
c_{r,0})
\chi_0'\chi_1\chi_2,
\end{split}\end{equation}
with
\begin{itemize}
\item
$f_{ij}$, $f'_i$, $f''_i$ and $f_0$ are all smooth
$\End(\Lambda X)$-valued
functions on $\Tb^* X\setminus o$, homogeneous of degree 0, satisfying
\eqref{eq:normal-f-estimates} (and are independent of $\digamma,\ep,\delta,r$),
\item
$\hat f_{ij}$, $\hat f'_i$, $\hat f''_i$ and $\hat f_0$ are all smooth
$\End(\Lambda X)$-valued
functions on $\Tb^* X\setminus o$, homogeneous of degree 0, with
$\hat f_{ij}=\hat f^\sharp_{ij}+\hat f^\flat_{ij}$,
$\hat f^\sharp_{ij},(\hat f'_i)^\sharp,(\hat f''_i)^\sharp,
\hat f_0^\sharp$ satisfying
\eqref{eq:normal-hat-f-sharp-estimates}
(and are independent of $\digamma,\ep,\delta,r$),
while $\hat f^\flat_{ij},(\hat f'_i)^\flat,(\hat f''_i)^\flat,
\hat f_0^\flat$ is supported away from $\dot\Sigma$,
\item
and $\hat c_{r,ij}$, $\hat c'_{r,i}$,
$\hat c''_{r,i}$ and $\hat c_{r,0}$ are all smooth
$\End(\Lambda X)$-valued
functions on $\Tb^* X\setminus o$, homogeneous of degree 0, uniformly bounded
in $\ep,\delta,r,\digamma$.
\end{itemize}

In fact, it is
useful to rewrite the leading term in $Q_i^*C_{r,ij}Q_j$,
namely the term that contributes to $C_{r,ij}$ with symbol
$-4w_r^2 \digamma^{-1}\delta^{-1}A_{ij}a
|\zeta_{n-k}|^{-1}\chi_0'\chi_1\chi_2$ (cf.\ \eqref{eq:normal-main-term}),
as a b-operator
using the PDE, modulo lower order terms.
Thus, let
$$
b_r=2w_r|\zeta_{n-k}|^{1/2}(\digamma\delta)^{-1/2}(\chi_0\chi'_0)^{1/2}\chi_1\chi_2,
$$
and let $\tilde B_r\in\Psib^{s+1}(X;\Lambda X)$ with principal symbol
$b_r\,\Id_{\Lambda X}$.
Then let
\begin{equation*}
C\in\Psib^0(X;\Lambda X),\ \sigma_{\bl,0}(C)=|\zeta_{n-k}|^{-1}
\tilde p^{1/2}\psi\,\Id_{\Lambda X}
\end{equation*}
where
$\psi\in S^0_{\hom}(\Tb^*X\setminus o)$ is identically $1$ on $U$
considered as a subset
of $\Sb^*X$; recall from Remark~\ref{rem:normal-remark}
that $\tilde p$ is bounded below by a positive
quantity here.

If $\tilde C_{r}\in\Psib^{2s}(X;\Lambda X)$
with principal symbol
$$
\sigma_{\bl,2s}(\tilde C_r)=
-4w_r^2 \digamma^{-1}\delta^{-1}a |\zeta_{n-k}|^{-1}\chi_0'\chi_1\chi_2
\,\Id_{\Lambda X}
=-|\zeta_{n-k}|^{-2}b_r^2\,\Id_{\Lambda X}
$$
and with $\tilde C_r^*$ preserving
boundary conditions\footnote{But not necessarily $C_r$; we can construct
$C_r$ by first constructing an operator preserving the boundary conditions
and with symbol equal the the adjoint of the desired symbol of $C_r$,
and then taking its adjoint.},
then, with $\sim$ denoting operators differing by an element
of $\Diffs^2\Psib^{2s-1}(X;\Lambda X)$,
\begin{equation*}\begin{split}
\sum_{ij}Q_i^* \tilde C_{r}A_{ij}Q_j&\sim\tilde C_{r} \sum_{ij}Q_i^* A_{ij}Q_j
=\tilde C_{r}(P-\sum_i(M_iQ_i+Q_i^*M'_i)-\tilde P)\\
&\sim\tilde C_r P-B_r^*\tilde B_r(\sum_i(M_iQ_i+Q_i^*M'_i))
+\tilde B_r^* C^*C\tilde B_r,
\end{split}\end{equation*}
so we deduce from
\eqref{eq:gend-commutator-normal}-\eqref{eq:gend-commutator-normal-symbol-2}
that\footnote{The $f^\sharp_{ij}$ terms are included in $R_{ij}$,
while the $f^\flat_{ij}$ terms are included in $E'$, and similarly
for the other analogous terms in $f'_i$, $f''_i$, $f_0$.}
\begin{equation}\begin{split}\label{eq:P-comm}
&\imath(A^\dagger_r A_r)^* P-\imath P A^\dagger_r A_r\\
&\quad=R'P+\tilde B^*_r\big(C^*C+R_0+\sum_i (Q_i^* R_i+\tilde R_i Q_i)
+\sum_{ij} Q_i^*R_{ij}Q_j\big)\tilde B_r+R''+E+E'
\end{split}\end{equation}
with
\begin{equation*}\begin{split}
&R_0\in\Psib^0(X;\Lambda X),\ R_i,\tilde R_i\in\Psib^{-1}(X;\Lambda X),
\ R_{ij}\in\Psib^{-2}(X;\Lambda X),\\
&R'\in\Psib^{-1}(X;\Lambda X),\ R''\in\Diff^2\Psib^{-2}(X;\Lambda X),
\ E,E'\in\Diff^2\Psib^{-1}(X;\Lambda X),
\end{split}\end{equation*}
with $\WFb'(E)\subset\eta^{-1}((-\infty,-\delta])
\cap U$, $\WFb'(E')\cap\dot\Sigma=\emptyset$,
$(R')^*$ preserving the boundary conditions,
and with
$r_0=\sigma_{\bl,0}(R_0)$, $r_i=\sigma_{\bl,-1}(R_i)$,
$\tilde r_i=\sigma_{\bl,-1}(\tilde R_i)$,
$r_{ij}\in\sigma_{\bl,-2}(R_{ij})$, $|.|$ denoting endomorphism norms,
\begin{equation*}\begin{split}
&|r_0|\leq C_2(\delta\ep+\ep^{-1}+\delta\digamma^{-1}),
\ |\zeta_{n-k} r_i|\leq C_2(\delta\ep+\ep^{-1}+\delta\digamma^{-1}),\\
&|\zeta_{n-k} \tilde r_i|\leq C_2(\delta\ep+\ep^{-1}+\delta\digamma^{-1}),
\ |\zeta_{n-k}^2 r_{ij}|\leq C_2(\delta\ep+\ep^{-1}+\delta\digamma^{-1}).
\end{split}\end{equation*}
This is exactly the form-valued version of
\cite[Equation~(6.18)]{Vasy:Propagation-Wave}, except the presence of the
$\delta\digamma^{-1}$ term which however is treated like the $\ep\delta$ term
for $\digamma$ sufficiently large, hence the rest of the
proof proceeds exactly as in that paper -- the only point where one
needs to use the boundary conditions from here on is that $(R')^*$
preserves these,
so $\langle Pu,(R')^*u\rangle=\langle f,(R')^*u\rangle$. In order to eliminate
duplication of the rest of the
argument of \cite[Proof of Proposition~6.2]{Vasy:Propagation-Wave},
we do not repeat it here.
\end{proof}

\section{Tangential propagation}\label{sec:tangential}
We now consider tangential propagation.

\begin{prop}\label{prop:tgt-prop}(Tangential propagation.)
Suppose that $P$ is as in \eqref{eq:P-form}, i.e.\ consider $P$
with relative boundary conditions.
Let $\cU_0$ be a coordinate chart in $X$, $\cU$ open with $\overline{\cU}
\subset\cU_0$.
Let $u\in H^1_{R,\loc}(X;\Lambda X)$, and let $\tilde\pi:T^*X\to T^*F$ be
the coordinate projection
$$
\tilde\pi:(x,y,\xi,\zeta)\mapsto(y,\zeta).
$$
Given $K\subset\Sb^*_{\cU}X$ compact
with
\begin{equation}
K\subset(\cG\cap \Tb^*_{F_{\reg}} X)\setminus\WFb^{-1,\infty}(f),\ f=P u,
\end{equation}
there exist constants $C_0>0$,
$\delta_0>0$ such that the following holds. If
$q_0=(y_0,\zeta_0)\in K$, $\alpha_0=\hat\pi^{-1}(q_0)$,
$W_0=\tilde\pi_*|_{\alpha_0}\sH_p$
considered as a constant vector field in local
coordinates, and
for some $0<\delta<\delta_0$, $C_0\delta\leq\epsilon<1$ and
for all $\alpha=(x,y,\xi,\zeta)\in \Sigma$
\begin{equation}\begin{split}\label{eq:tgt-prop-est}
\alpha\in T^*X &\Mand
|\tilde\pi(\alpha-\alpha_0-\delta W_0)|
\leq\epsilon\delta\Mand
|x(\alpha)|\leq\epsilon\delta\\
&\Rightarrow
\pi(\alpha)\notin\WFb^{1,\infty}(u),
\end{split}\end{equation}
then $q_0\notin\WFb^{1,\infty}(u)$.
\end{prop}

\begin{proof}
Again, we first construct the symbol $a$ of our commutator following
\cite[Proof of Proposition~7.3]{Vasy:Propagation-Wave} as corrected
in \cite{Vasy:Propagation-Wave-Correction}.
Note that
$$
W_0(q_0)=\sH_{\tilde p}(q_0),
$$
and let
$$
W=|\zeta_{n-k}|^{-1}W_0,
$$
so $W$ is homogeneous of degree zero (with respect to the $\RR^+$-action on
the fibers of $T^*F\setminus o$).
We can use
$$
\tilde\eta=(\sgn(\zeta_{n-k})_0)(y_{n-k}-(y_{n-k})_0)
$$
now to measure propagation,
since $\zeta_{n-k}^{-1}\sH_{\tilde p}(y_{n-k})=2>0$ at $q_0$, so
$\sH_{\tilde p}\tilde\eta$ is
$2|\zeta_{n-k}|>0$ at $q_0$.

First, we require
$$
\rho_1=\tilde p(y,\zetah)=|\zeta_{n-k}|^{-2}\tilde p(y,\zeta);
$$
note that $d\rho_1\neq 0$ at $q_0$ for $\zeta\neq 0$ there, but
$\sH_{\tilde p}\tilde p\equiv 0$, so
$$
W\rho_1(q_0)=0.
$$
Next, since
$\dim F=n-k$, $\dim T^*F=2n-2k$, hence $\dim S^*F=2n-2k-1$. With a
slight abuse of notation, we also regard $q_0$ as a point in
$S^*F$ -- recall that $S^*F=(T^*F\setminus o)/\RR^+$. We can also
regard $W$ as a vector field on $S^*F$ in view of its homogeneity.
As $W$ does not vanish as a vector in $T_{q_0}S^*F$
in view of
$W\tilde\eta(q_0)\neq 0$, $\tilde\eta$ being
homogeneous degree zero, hence a function on $S^*F$,
the kernel of $W$ in $T^*_{q_0}S^*F$ has dimension $2n-2k-2$.
Thus there exist
$\rho_j$, $j=2,\ldots,2n-2k-2$ be homogeneous degree zero functions
on $T^*F$ (hence functions on $S^*F$)
such that
\begin{equation}\begin{split}
&\rho_j(q_0)=0,\ j=2,\ldots,2n-2k-2,\\
&W\rho_j(q_0)=0,\ j=2,\ldots,2n-2k-2,\\
&d\rho_j(q_0),\ j=1,\ldots,2n-2k-2\ \text{are linearly independent at}\ q_0.\\
\end{split}\end{equation}
By dimensional considerations,
$d\rho_j(q_0)$, $j=1,\ldots,2n-2k-2$, together with $d\tilde\eta$
span the cotangent space of $S^*F$ at $q_0$, i.e.\ of
the quotient of $T^*F$ by the $\Real^+$-action.

Hence,
$$
|\zeta_{n-k}|^{-1}W_0 \rho_j=\sum_{i=1}^{2n-2k-2}\tilde F_{ij}\rho_j+
\tilde F_{2n-2k-1}\tilde\eta,\ j=2,\ldots,2n-2k-2,
$$
with $\tilde F_{ij}$ smooth, $i=1,\ldots,2n-2k-2$, $j=2,\ldots,2n-2k-2$.
Then we
extend $\rho_j$ to a function on $\Tb^*X\setminus o$ (using the coordinates
$(x,y,\sigma,\zeta)$), and conclude that
\begin{equation}\label{eq:sH-rho_2}
|\zeta_{n-k}|^{-1}\sH_{\tilde p} \rho_j=
\sum_{l=1}^{2n-2k-2}\tilde F_{jl}\rho_l+
\tilde F_{2n-2k-1}\tilde\eta+\sum_l \tilde F_{0,jl} x_l,\ j=2,\ldots,2n-2k-2,
\end{equation}
with $\tilde F_{jl}$, $\tilde F_{0,jl}$ smooth, $\pa_{x_l}\rho_j$ vanishes.
Similarly,
\begin{equation}\label{eq:sH-tilde-eta}
|\zeta_{n-k}|^{-1}\sH_{\tilde p} \tilde\eta=2+\sum_{l=1}^{2n-2k-2}\hat F_l\rho_l
+\hat F_{2n-2k-1}\tilde\eta
+\sum_l \hat F_{0,l} x_l,
\end{equation}
with $\hat F_l$, $\hat F_{0,l}$ smooth,
$\pa_{x_l}\tilde\eta$ vanishes, as do vertical
derivatives of $\tilde\eta$.

Let
\begin{equation}
\omega=|x|^2+\sum_{j=1}^{2n-2k-2}\rho_j^2.
\end{equation}
Finally, we let
\begin{equation}\label{eq:glancing-phi-def}
\phi=\tilde\eta+\frac{1}{\ep^2\delta}\omega,
\end{equation}
and define $a$ by
\begin{equation}\label{eq:prop-22t}
a=\chi_0(\digamma^{-1}(2-\phi/\delta))\chi_1((\tilde\eta\delta)/
\ep\delta+1)\chi_2(|\sigma|^2/\zeta_{n-k}^2),
\end{equation}
with $\chi_0,\chi_1$ and $\chi_2$ as in the case of the normal
propagation estimate, stated after \eqref{eq:normal-phi-def}.
We always assume $\ep<1$, so on $\supp a$ we have
\begin{equation*}
\phi\leq 2\delta\Mand \tilde\eta\geq-\ep\delta-\delta\geq-2\delta.
\end{equation*}
Since $\omega\geq 0$,
the first of these inequalities implies that
$\tilde\eta\leq 2\delta$, so on $\supp a$
\begin{equation}
|\tilde\eta|\leq 2\delta.
\end{equation}
Hence,
\begin{equation}\label{eq:omega-delta-est-t}
\omega\leq \ep^2\delta(2\delta-\tilde\eta)\leq4\delta^2\ep^2.
\end{equation}
Moreover, on $\supp d\chi_1$,
\begin{equation}\label{eq:dchi_1-supp}
\tilde\eta\in[-\delta-\ep\delta,-\delta],\ \omega^{1/2}
\leq 2\ep\delta,
\end{equation}
so this region lies in \eqref{eq:tgt-prop-est} after $\ep$ and $\delta$
are both replaced by appropriate constant multiples, namely the present
$\delta$ should be replaced by $\delta/(2|(\zeta_{n-k})_0|)$.

We proceed as in the case of hyperbolic points,
letting $A_0\in\Psib^0(X)$ with $\sigma_{\bl,0}(A_0)=a$, supported
in the coordinate chart, and letting
$A=A_0\otimes\Id$ with respect to the trivialization.
Also let
$\Lambda_r$ be scalar with symbol
\begin{equation}
|\zeta_{n-k}|^{s+1/2}(1+r|\zeta_{n-k}|^2)^{-s}\,\Id,\quad r\in[0,1),
\end{equation}
so $A_r=A\Lambda_r\in\Psib^{0}(X;\Lambda X)$ for $r>0$ and it is
uniformly bounded in $\Psibc^{s+1/2}(X;\Lambda X)$.
Then, for $r>0$,
\begin{equation}\label{eq:pairing-identity-glance}
\langle \imath Pu,A^\dagger_r A_r u\rangle-\langle \imath A^\dagger_r A_r  u,Pu\rangle
=\langle \imath(A^\dagger_r A_r)^* Pu,u\rangle
-\langle \imath P  A^\dagger_r A_r u,u\rangle.
\end{equation}
We can compute this using Proposition~\ref{prop:gend-comm}, noting
that the vector fields $\tilde V_{ij}$, $\tilde V_i$ and $\tilde V_0$
satisfy
$$
\tilde V_\bullet \chi(\digamma^{-1}(2-\phi/\delta))=
\digamma^{-1}\ep^{-2}\delta^{-2}
(\tilde V_\bullet\omega)\cdot\chi'_0(\digamma^{-1}(2-\phi/\delta)),
$$
since $\tilde\eta$ is the pull-back of a function on the
base and $\tilde V_\bullet$ is vertical.

We arrange
the terms of the proposition so that the terms in which a vector
field differentiates $\chi_1$ are included in $E_r$, the terms in which
a vector fields differentiates $\chi_2$ are included in $E'_r$. Thus, we
have
\begin{equation}\label{eq:gend-commutator-tgt}
\imath(A^\dagger_r A_r)^* P-\imath P A^\dagger_r A_r =
Q_i^* C_{r,ij}Q_j+Q_i^* C'_{r,i}+C''_{r,j}Q_j+C_{r,0}+E_r+E'_r+F_r,
\end{equation}
with
\begin{equation}\begin{split}\label{eq:gend-commutator-tgt-symbols}
&\sigma_{\bl,2s}(C_{r,ij})=w_r^2\Big(\digamma^{-1}\delta^{-1}a
|\zeta_{n-k}|^{-1}\big(f_{1,ij}+\ep^{-2}\delta^{-1}
f_{0,ij}\big)
\chi_0'\chi_1\chi_2+a^2 \tilde c_{r,ij}\Big),\\
&\sigma_{\bl,2s+1}(C_{r,i}^\bullet)=w_r^2\Big(\digamma^{-1}\delta^{-1}a
\big(f_{1,i}^\bullet+\delta^{-1}\ep^{-2}f_{0,i}^\bullet\big)
\chi_0'\chi_1\chi_2+a^2 \tilde c_{r,i}^\bullet\Big),\ \bullet=',''\\
&\sigma_{\bl,2s+2}(C_{r,0})
=w_r^2\Big(|\zeta_{n-k}|\digamma^{-1}\delta^{-1}a\big(4+f_{1,0}
+\delta^{-1}
\ep^{-2} f_{0,0}\big)
\chi_0'\chi_1\chi_2+a^2 \tilde c_{r,0}\Big),
\end{split}\end{equation}
where
$f_{k,ij}$, $f_{k,i}$ and $f_{k,0}$ are all smooth $\End(\Lambda X)$-valued
functions on $\Tb^* X\setminus o$,
homogeneous of degree 0 (independent of
$\ep\in[0,1)$, $k=0,1$, and $\delta\in[0,1]$), arising
when $\omega$ is differentiated in $\chi_0(\cte^{-1}(2-\phi/\delta))$
for $k=0$ and when $\tilde\eta$ is differentiated for $k=1$, and
all such terms are included in these except $f_{1,0}$ is missing
$|\zeta_{n-k}|^{-1}\sH_{\tilde p}\tilde\eta(q_0)=2>0$ extended as the constant
function near $q_0$. Moreover, as $V_\bullet \rho^2=2\rho V_\bullet\rho$
for any function $\rho$, the terms with $k=0$ have vanishing factors
of $\rho_l$, resp.\ $x_l$, with the structure of
the remaining factor dictated by the form of $V_\bullet\rho_l$,
resp.\ $V_\bullet x_l$, and similarly
for $\tilde V$ (which only affects the $k=0$ terms as discussed
earlier).
Thus, using \eqref{eq:sH-rho_2} to compute $f_{0,0}$,
\eqref{eq:sH-tilde-eta} to compute $f_{1,0}$, and recalling that
$\tilde V$ appears in the form $\tilde p\tilde V=\rho_1\tilde V$
in $f_{0,0}$,
\begin{equation*}\begin{split}
f_{0,ij}&=\sum_{k} \rho_k \tilde f_{0,ijk}+\sum_{k}x_k
\hat f_{0,ijk},\\
f_{0,i}^\bullet&=\sum_l \rho_k \tilde f_{0,ik}^\bullet
+\sum_k x_k\hat f_{0,ik}^\bullet,\ \bullet=','',\\
f_{0,0}&=\sum_{kl}\rho_k\rho_l\tilde f_{0,0kl}
+\sum_{kl} \rho_k x_l\hat f_{0,0kl}+\sum_{kl} x_k x_l \check f_{0,0kl}
\sum_{k} \rho_k\tilde\eta f^\sharp_{0,0k}+\sum_k x_k\tilde\eta
f^\flat_{0,0k},\\
f_{1,0}&=\sum_k x_k \tilde f_{1,0k}+\sum_k\rho_k \hat f_{1,0k}+
\tilde\eta f^\flat_{1,0},
\end{split}\end{equation*}
with $\tilde f_{0,ijk}$, etc., smooth.
With $|.|$ denoting norms in $\End(\Lambda X)$, we deduce that
\begin{equation}\begin{split}\label{eq:tgt-prop-quad-terms}
\ep^{-2}\delta^{-1}|f_{0,ij}|\leq C\ep^{-1},
\ |f_{1,ij}|\leq C,
\end{split}\end{equation}
while
\begin{equation}\label{eq:tgt-prop-lin-terms}
\ep^{-2}\delta^{-1}|f_{0,i}^\bullet|\leq C\ep^{-1},
\ |f_{1,i}^\bullet|\leq C,
\end{equation}
$\bullet=',''$, and
\begin{equation}\label{eq:tgt-prop-tgt-terms}
\ep^{-2}\delta^{-1}|f_{0,0}|\leq C\ep^{-1}\delta,\ |f_{1,0}|\leq C\delta.
\end{equation}
We remark that although thus far we worked with a single $q_0\in K$, the
same construction works with $q_0$ in a neighborhood $U_{q'_0}$ of a fixed
$q'_0\in K$, with a {\em uniform} constant $C$. In view of the
compactness of $K$, this suffices (by the rest
of the argument we present below) to give the uniform estimate
of the proposition.

For a small constant $c_0>0$ to be determined, which we may assume to be less
than $C$, we demand below that
the expressions on the right hand sides of \eqref{eq:tgt-prop-quad-terms}
are bounded by $c_0(\ep\delta)^{-1}$,
those on the right hand sides of
\eqref{eq:tgt-prop-lin-terms} are bounded by $c_0(\ep\delta)^{-1/2}$,
while those on the right hand sides of \eqref{eq:tgt-prop-tgt-terms}
are bounded by $c_0$. This demand is due to the appearance of two, resp.\ one,
resp.\ zero, factors of $Q_i$ in \eqref{eq:gend-commutator-tgt} for
the terms whose principal symbols are affected by these, taking
into account that in view of Lemma~\ref{lemma:Dt-Dx}
we can estimate $\|Q_i v\|$ by
$C_{\cG,K}(\ep\delta)^{1/2}\|(D_{y_{n-k}}\otimes\Id)v\|$
if $v$ is microlocalized to a $\ep\delta$-neighborhood of
$\cG$, which
is the case for us with $v=A_r u$ in terms of support properties of
$a$.

Thus, we need\footnote{In the sense that if these hold, the right hand
side of \eqref{eq:tgt-prop-quad-terms}-\eqref{eq:tgt-prop-tgt-terms}
satisfy the desired estimates.}
\begin{equation}\begin{split}\label{eq:demanded-estimates}
&C\ep^{-1}\leq c_0\delta^{-1}\ep^{-1},
\ C\leq c_0\delta^{-1}\ep^{-1};\\
&C\ep^{-1}\leq c_0\delta^{-1/2}\ep^{-1/2},
\ C\leq c_0\delta^{-1/2}\ep^{-1/2};\\
&C\ep^{-1}\delta\leq c_0,\ C\delta\leq c_0;
\end{split}\end{equation}
here the semicolons correspond to the breakup corresponding to
the various lines of
\eqref{eq:tgt-prop-quad-terms}-\eqref{eq:tgt-prop-tgt-terms}.
By the first inequality on the last line, we must have
$\ep\geq (C/c_0)\delta$, by the first equation on the second line
$\ep\geq (C/c_0)^2\delta$, which implies the preceeding equation
as $c_0<C$. Assuming
\begin{equation}\label{eq:ep-estimate}
(C/c_0)^2\delta\leq\ep\leq 1,
\end{equation}
all the
preceeding equations hold for sufficiently small $\delta$, namely
\begin{equation}\label{eq:delta-estimate}
\delta<(c_0/C)^2,
\end{equation}
as we check below.

Thus, with $\ep,\delta$ satisfying \eqref{eq:ep-estimate} and
\eqref{eq:delta-estimate}, hence
$\delta^{-1}>(C/c_0)^2>C/c_0$,
\eqref{eq:tgt-prop-quad-terms}-\eqref{eq:tgt-prop-tgt-terms} give that
\begin{equation}\begin{split}\label{eq:tgt-prop-quad-terms-rev}
\ep^{-2}\delta^{-1}|f_{0,ij}|\leq C\ep^{-1}\leq c_0\delta^{-1}\ep^{-1},
\ |f_{1,ij}|\leq C\leq c_0\delta^{-1}\leq c_0\delta^{-1}\ep^{-1},
\end{split}\end{equation}
while
\begin{equation}\label{eq:tgt-prop-lin-terms-rev}
\ep^{-2}\delta^{-1}|f_{0,i}^\bullet|\leq C\ep^{-1}
\leq c_0\delta^{-1/2}\ep^{-1/2},
\ |f_{1,i}^\bullet|\leq C\leq c_0\delta^{-1/2}\leq c_0\delta^{-1/2}\ep^{-1/2},
\end{equation}
$\bullet=',''$, and
\begin{equation}\label{eq:tgt-prop-tgt-terms-rev}
\ep^{-2}\delta^{-1}|f_{0,0}|\leq C\ep^{-1}\delta\leq c_0,
\ |f_{1,0}|\leq C\delta\leq c_0.
\end{equation}

Again, the terms with $a^2$ in \eqref{eq:gend-commutator-tgt}
are negligible, for they can all be rewritten using
\eqref{eq:chi_0-chi_0-prime}.

Let $\tilde B_r\in\Psib^{s+1}(X;\Lambda X)$ with
$\sigma_{\bl,0}(\tilde B_r)=\tilde b_r\,\Id$,
\begin{equation*}
\tilde b_r=2w_r|\zeta_{n-k}|^{1/2}
(\cte\delta)^{-1/2}(\chi_0\chi_0')^{1/2}\chi_1\chi_2\in\Cinf(\Tb^*X\setminus o).
\end{equation*}
Again,
$\chi_0$ stands for $\chi_0(\cte^{-1}(2-\frac{\phi}{\delta}))$, etc.
Also, let $C\in\Psib^0(X;\Lambda X)$
have symbol $\sigma_{\bl,0}(C)=\psi\,\Id$ where
$\psi\in S^0_{\hom}(\Tb^*X\setminus o)$ is
identically $1$ on $U$ considered as a subset
of $\Sb^*X$.
Then
\begin{equation*}\begin{split}
&\imath(A^\dagger_r A_r)^* P-\imath P A^\dagger_r A_r \\
&\quad=\tilde B^*_r(C^*C+R_0+\sum_i (Q_i^* R_i+\tilde R_i Q_i)
+\sum_{ij} Q_i^* R_{ij} Q_j)\tilde B_r+R''+E+E'
\end{split}\end{equation*}
with
\begin{equation*}\begin{split}
&R_0\in\Psib^0(X;\Lambda X),\ R_i,\tilde R_i\in\Psib^{-1}(X;\Lambda X),
\ R_{ij}\in\Psib^{-2}(X;\Lambda X),\\
&
\ R''\in\Diff^2\Psib^{-2}(X;\Lambda X),
\ E,E'\in\Diff^2\Psib^{-1}(X;\Lambda X),
\end{split}\end{equation*}
with $\WFb'(E)\subset\tilde\eta^{-1}((-\infty,-\delta])
\cap U$, $\WFb'(E')\cap\dot\Sigma=\emptyset$,
and with
$r_0=\sigma_{\bl,0}(R_0)$, $r_i=\sigma_{\bl,-1}(R_i)$,
$\tilde r_i=\sigma_{\bl,-1}(\tilde R_i)$,
$r_{ij}\in\sigma_{\bl,-2}(R_{ij})$, satisfying
\begin{equation}\begin{split}\label{eq:r_0-r_ij-est}
&|r_0|\leq 2c_0
+C_2\delta\digamma^{-1},\\
&|\zeta_{n-k} r_i|,|\zeta_{n-k} \tilde r_i|\leq
2c_0\delta^{-1/2}\ep^{-1/2}+C_2\delta\digamma^{-1},\\
&|\zeta_{n-k}^2 r_{ij}|\leq 2c_0\delta^{-1}\ep^{-1}
+C_2\delta\digamma^{-1}.
\end{split}\end{equation}
Here the $C_2\delta\digamma^{-1}$ terms arise
by incorporating the $a^2$-terms of \eqref{eq:gend-commutator-tgt-symbols},
using \eqref{eq:chi_0-chi_0-prime}, as in the normal case.

These are exactly the formed-valued versions of
the result of the second displayed equation after
\cite[Equation~(7.16)]{Vasy:Propagation-Wave}, as corrected in
\cite{Vasy:Propagation-Wave-Correction}, with the small (at this point
arbitrary) constant $c_0$ replacing some constants given there
in terms of $\ep$ and $\delta$: in \cite{Vasy:Propagation-Wave-Correction}.
Thus,
the rest of the argument thus proceeds as in
\cite[Proof of Proposition~7.3]{Vasy:Propagation-Wave}, taking into
account \cite{Vasy:Propagation-Wave-Correction}. For example, the
estimate on $\tilde R_i$ takes the following form.

We let, as in \cite{Vasy:Propagation-Wave},
$T\in\Psib^{-1}(X;\Lambda X)$ be elliptic with principal symbol $|\zeta_{n-k}|^{-1}$
on a neighborhood of $\supp a$,
$T^-\in\Psib^1(X;\Lambda X)$ a parametrix, so
$T^-T=\Id+F$, $F\in\Psib^{-\infty}(X;\Lambda X)$.
In view of \eqref{eq:r_0-r_ij-est} there exists
$R'_i\in\Psib^{-1}(X;\Lambda X)$ such that for any $\gamma>0$,
\begin{equation*}\begin{split}
\|\tilde R_i w\|&=\|\tilde R_i (T^- T -F)w\|\leq\|(\tilde R_i T^-)(Tw)\|
+\|\tilde R_iFw\|\\
&\leq 2(2c_0\delta^{-1/2}\ep^{-1/2}+C_2\delta\digamma^{-1})
\|Tw\|+
\|R_i' Tw\|+\|\tilde R_iFw\|
\end{split}\end{equation*}
for all $w$ with $Tw\in L^2(X;\Lambda X)$, hence
\begin{equation*}\begin{split}
|\langle \tilde R_i Q_i v,v\rangle|\leq
&2(2c_0\delta^{-1/2}\ep^{-1/2}+C_2\delta\digamma^{-1})
\|TQ_i v\|\,\|v\|\\
&\qquad+2\gamma\|v\|^2+\gamma^{-1}\|R'_i TQ_i v\|^2+\gamma^{-1}\|F_i Q_iv\|^2,
\end{split}\end{equation*}
with $F_i\in\Psib^{-\infty}(X;\Lambda X)$.
Now we use that $a$ is microlocalized
in an $\ep\delta$-neighborhood of $\cG$, hence the same can
be arranged for $T$:
$\cG$ is given by $\rho_1=0$, $x=0$, and
we are microlocalized to the region where $|\rho_1|\leq 2\ep\delta$,
$|x|\leq 2\ep\delta$.
For $v=\tilde B_r u$,
Lemma~\ref{lemma:Dt-Dx} thus gives
(taking into account that we need to estimate
$\|TQ_i v\|$ rather than its square)
\begin{equation*}\begin{split}
&|\langle \tilde R_i Q_i v,v\rangle|\leq 4C_{\cG,K}
(2c_0\delta^{-1/2}\ep^{-1/2}+C_2\delta\digamma^{-1})
(\ep\delta)^{1/2}\|\tilde B_r u\|^2\\
&\qquad\qquad+C_0\gamma^{-1}\big(\|G\tilde B_r u\|_{H^1}^2
+\|\tilde B_r u\|_{H^1_{\loc}}^2+\|\tilde G Pu\|^2_{\dot H^{-1}}
+\|Pu\|^2_{H^1_\loc}\big)\\
&\qquad\qquad+3\gamma\|\tilde B_r u\|^2+\gamma^{-1}\|R'_i TD_{x_i} \tilde B_r u\|^2+\gamma^{-1}\|F_i D_{x_i}\tilde B_r u\|^2.
\end{split}\end{equation*}
The first term is the main term of interest, and
its coefficient satisfies
$$
4C_{\cG,K}(2c_0\delta^{-1/2}\ep^{-1/2}+C_2\delta\digamma^{-1})
(\ep\delta)^{1/2}\leq C_3(c_0+\delta\digamma^{-1})
$$
with some $C_3$ depending only on $C_{\cG,K},C,C_2$. We can proceed as
in \cite{Vasy:Propagation-Wave} if given a prescribed
quantity $c_1>0$ we can find $\delta_0>0$ such that this coefficient
is less than $c_1$ for $\delta\in(0,\delta_0)$
(and can do the same with analogous terms for $R_{ij}$, $R_i$ and $R_0$, which
however are easily handled by a similar argument).
But we can indeed achieve this by first choosing $c_0$ sufficiently
small,
then $\delta_0$ sufficiently small according \eqref{eq:delta-estimate}
(i.e.\ take $\delta_0<(c_0/C)^2$), finally
$\digamma$ sufficiently large. The proof is
thus finished as in
in \cite[Proof of Proposition~7.3]{Vasy:Propagation-Wave}, thus completing
the proof.

\end{proof}

\section{Propagation of singularities}\label{sec:prop-sing}
Recall from \eqref{eq:P-form} that
we assume that
\begin{equation}\begin{split}
&P=\Box+P_1:H^1_{R,\loc}(X;\Lambda X)\to \dot H^{-1}_{R,\loc}(X;\Lambda X),\\
&P_1\in\Diff^1(X;\Lambda X)+\Diffd^1(X;\Lambda X).
\end{split}\end{equation}
The theorem on propagation of singularities is the following.

\begin{thm}[Slightly strengthened restatement of Theorem~\ref{thm:prop-sing}]
\label{thm:prop-sing-restate}
Suppose that $P$ is as in \eqref{eq:P-form}, i.e.\ consider $P$
with relative boundary conditions, $Pu=f$.
If $u\in H^{1}_{R,\loc}(X;\Lambda X)$,
then for all $s\in\RR^+\cup\{\infty\}$ (with the convention $\infty+1=\infty$),
$$
\WFb^{1,s}(u)\setminus\WFb^{-1,s+1}(f)\subset\dot\Sigma,
$$
and it
is a union of maximally extended
generalized broken bicharacteristics of $P$ in
$\dot\Sigma\setminus\WFb^{-1,s+1}(Pu)$.

The same conclusion holds with relative boundary conditions replaced by
absolute boundary conditions.
\end{thm}

\begin{proof}
The proof proceeds as in \cite[Proof of Theorem~8.1]{Vasy:Propagation-Wave},
since the Propositions \ref{prop:normal-prop}
and \ref{prop:tgt-prop} are complete
analogues of \cite[Proposition~6.2]{Vasy:Propagation-Wave}
and \cite[Proposition~7.3]{Vasy:Propagation-Wave}.
Given the results of the previous sections,
this argument itself is only a slight modification of an argument originally
due to Melrose and Sj\"ostrand \cite{Melrose-Sjostrand:I}, as presented
by Lebeau \cite{Lebeau:Propagation}.
\end{proof}

One can relax the hypotheses of this Theorem in order to allow solutions
of $Pu=f$ with negative order of $\bl$-regularity relative to
$H^1_{R,\loc}(X;\Lambda X)$. This is of importance because this way
one can deal with the (say, forward) fundamental solution of the
wave equation directly.

The b-Sobolev spaces with negative order relative
to $H^1(X)$ and $\dot H^{-1}(X)$ were defined
in \cite[Definition~3.15]{Vasy:Propagation-Wave}; we refer
to Section~3 of \cite{Vasy:Propagation-Wave} for most details. Here we first
state the bundle valued analogue:

\begin{Def}\label{Def:H1m-neg}
Let $E$ be a vector bundle over $X$.
Let $m<0$, and $A\in\Psib^{-m}(X;E)$ be elliptic on $\Sb^*X$
with proper support.
We let $H^{1,m}_{\bl,\compl}(X;E)$
be the space of all $u\in\ddist(X;E)$ of the
form $u=u_1+Au_2$ with $u_1,u_2\in H^1_\compl(X;E)$.
We let
\begin{equation*}\begin{split}
\|u\|_{H^{1,m}_{\bl,\compl}(X;E)}
=\inf\{\|u_1\|_{H^1(X;E)}
+\|u_2\|_{H^1(X;E)}:\ u=u_1+Au_2\}.
\end{split}\end{equation*}

We also let $H^{1,m}_{\bl,\loc}(X;E)$ be the space of all
$u\in\ddist(X;E)$
such that $\phi u\in H^{1,m}_{\bl,\compl}(X;E)$ for all
$\phi\in\Cinf_\compl(X)$.

We define $\dot H^{-1,m}_{\bl,\compl}(X;E)$ and
$\dot H^{-1,m}_{\bl,\loc}(X;E)$
analogously, replacing $H^1(X;E)$
$\dot H^{-1}(X;E)$
throughout the
above discussion.
\end{Def}

\begin{rem}\label{rem:scalar-definability}
For $u$ supported in a coordinate chart in which $E$ is trivialized,
without loss of generality we may require that $A$ is scalar
in that particular trivialization;
as shown in \cite[Section~3, following Remark~3.16]{Vasy:Propagation-Wave}
all choices of $A$
are equivalent (as long as they are elliptic on a neighborhood of
$\supp u$).

Indeed, this -- together with the locality of these spaces, i.e.\ that
they are preserved by multiplication by $\phi\in\CI(X)$ --
shows that $H^{1,m}_{\bl,\compl}(X;E)$ could also be defined
by localization, and requiring that in local coordinates in which $E$
is trivial, the sections are $N$-tuples of $H^{1,m}_{\bl,\compl}(X)$
functions with $N$ being the rank of $E$.
\end{rem}

The restriction map to a boundary hypersurface $\hsf$,
$\gamma_{\hsf}:\CI(X;E)\to \CI(\hsf;E_{\hsf})$ extends to
a map
\begin{equation*}\begin{split}
&\gamma_{\hsf}:
H^{1,m}_{\bl,\compl}(X;E)\to \ddist_{\compl}(\hsf;E_{\hsf}),\\
&\gamma_{\hsf}(u_1+Au_2)=\gamma_{\hsf}(u_1)+\hat N_{\hsf}(A)(0)
\gamma_{\hsf}(u_2);
\end{split}\end{equation*}
see \cite[Remark~3.16]{Vasy:Propagation-Wave}.
In particular, we make the following definition:

\begin{Def}\label{Def:H1m-neg-R}
If $E$ is a vector bundle over $X$, $m<0$,
$$
\dot H^{1,m}_{\bl,\compl}(X;E)=\{u\in H^{1,m}_{\bl,\compl}(X;E):
\ \forall\hsf\in\pa_1(X),\ \gamma_{\hsf}(u)=0\}
$$
and
\begin{equation*}\begin{split}
H^{1,m}_{\bl,R,\compl}(X;\Lambda X)=\{u\in & H^{1,m}_{\bl,\compl}(X;\Lambda X):
\\
& \forall\hsf\in\pa_1(X),\ \gamma_{\hsf}(u)\in\ddist(\hsf;\Lambda_{\hsf,N}X)\}.
\end{split}\end{equation*}

The local spaces are defined analogously.
\end{Def}

Equivalently, as follows from the corresponding
statements for $\dot H^1(X;E)=H^1_0(X;E)$ resp.\ $H^1_R(X;\Lambda X)$,
$\dot H^{1,m}_{\bl,\compl}(X;E)$
is the closure of $\dCI_{\compl}(X;E)$ and
$H^{1,m}_{\bl,R,\compl}(X;\Lambda X)$ is the closure
of $\CI_{R,\compl}(X;\Lambda X)$ in the $H^{1,m}$ topology.

Also, equivalently, $H^{1,m}_{\bl,R,\compl}(X;\Lambda X)$
is the space of all $u\in\ddist_{\compl}(X;\Lambda X)$ of the
form $u=u_1+Au_2$ with $u_1,u_2\in H^1_{R,\compl}(X;\Lambda X)$ if
$A$ satisfies \eqref{eq:preserve-normal-forms} (this can always be
assumed locally by Remark~\ref{rem:scalar-definability}). This follows by
considering a parametrix $G$ for $A$, $E=GA-\Id,F=AG-\Id
\in\Psib^{-\infty}(X;\Lambda X)$, such that
$G$ also satisfies \eqref{eq:preserve-normal-forms}. Then $u=u_1'+Au_2'$ with
$u_k'\in H^1_{\compl}(X;\Lambda X)$, so
\begin{equation*}\begin{split}
&u=AGu-Fu=A(Gu_1'+GAu_2')-Fu_1'-FAu_2'
=u_1+Au_2,\\
& u_1=-Fu=-Fu_1'-FAu_2',\ u_2=Gu=Gu_1'+GAu_2',
\end{split}\end{equation*}
where the first expression for $u_k$ shows $\gamma_{\hsf}(u_k)=0$,
and the second shows $u_k\in H^1_{\compl}(X;\Lambda X)$.

We still need to define the negative $\bl$-regularity version
of $\dot H^{-1}_R(X;\Lambda X)$. As $\dot H^{-1}_R(X;\Lambda X)$ is a quotient
space of $\dot H^{-1}(X;\Lambda X)$, just like
$\ddist_R(X;\Lambda X)$ (the dual of $\CI_R(X;\Lambda X)$) is a quotient
space of $\ddist(X;\Lambda X)$, we proceed as follows.

\begin{Def}\label{Def:H-1m-neg}
We let
$$
\dot H^{-1,m}_{\bl,R,\compl}(X;\Lambda X),\ \text{resp.}
\ \dot H^{-1,m}_{\bl,R,\loc}(X;\Lambda X),
$$
be the image of $\dot H^{-1,m}_{\bl,\compl}(X;\Lambda X)$,
resp.\ $\dot H^{-1,m}_{\bl,\loc}(X;\Lambda X)$,
in $\ddist_R(X;\Lambda X)$.
\end{Def}

If $A\in\Psibc(X;\Lambda X)$ with normal operators of $A^*$ satisfying
\eqref{eq:preserve-normal-forms} then
$$
A^*:\CI_R(X;\Lambda X)\to\CI_R(X;\Lambda X),
$$
so $A$ actually descends to the quotient space
$\ddist_R(X;\Lambda X)$,
cf.\ \eqref{eq:Psibc-dist}-\eqref{eq:Psibc-dist-def}.
We conclude the following:

\begin{lemma}\label{lemma:H-1R-form}
Let $m<0$, and $A\in\Psib^{-m}(X;\Lambda X)$ be elliptic on $\Sb^*X$
with proper support and with the normal operators $\hat N_{\hsf}(A^*)$
of $A^*$
satisfying \eqref{eq:preserve-normal-forms}.
Then $u\in \dot H^{-1,m}_{\bl,R,\compl}(X;\Lambda X)$ if and only if
$u=u_1+Au_2$ with $u_1,u_2\in \dot H^{-1}_{R,\compl}(X;\Lambda X)$.
\end{lemma}

\begin{proof}
If $u'\in\dot H^{-1,m}_{\bl,\compl}(X;\Lambda X)$ with the image of
$u'$ in $\ddist_R(X;\Lambda X)$ being $u$, then (with this $A$,
see Remark~\ref{rem:scalar-definability} on the independence of the
definition from the choice of $A$) $u'=u'_1+Au'_2$,
$u'_k\in \dot H^{-1}(X;\Lambda X)$. Let $u_k$ be the image of
$u'_k$ in $\dot H^{-1}_R(X;\Lambda X)\subset\ddist_R(X;\Lambda X)$;
the claim then follows immediately as $A$ acts on $\ddist_R(X;\Lambda X)$.
The converse follows by letting $u'_k\in \dot H^{-1}_{\compl}(X;\Lambda X)$
have image $u_k\in \dot H^{-1}_{R,\compl}(X;\Lambda X)$.
\end{proof}

In view of Definition~\ref{Def:Diff-Psib} and the following remarks,
namely that one can rearrange the order of factors in $\Diff(X)$ and
$\Psib(X)$ (changing the factors but not their (pseudo)differential
orders), without affecting
the principal symbol, and without affecting the mapping property
\eqref{eq:preserve-normal-forms} of normal operators (due to the
explicit form of $\tilde A$ and $\tilde B$ in \eqref{eq:A-At-Bt};
by Lemma~\ref{lemma:H-1R-form}
we presently need that the adjoints satisfy \eqref{eq:preserve-normal-forms}),
we deduce that any $Q\in\Diff^2(X;\Lambda X)$ gives
a map
$$
Q:H^{1,m}_{\bl,\loc}(X;\Lambda X)\to H^{-1,m}_{\bl,R,\loc}(X;\Lambda X).
$$
In particular, this is the case for
$P$ as in \eqref{eq:P-form}, and thus
\begin{equation}\label{eq:P-form-neg}
P:H^{1,m}_{\bl,R,\loc}(X;\Lambda X)\to H^{-1,m}_{\bl,R,\loc}(X;\Lambda X).
\end{equation}

We also recall from \cite[Section~3]{Vasy:Propagation-Wave} the
wave front set with negative order of regularity relative to $H^1$ and
$\dot H^{-1}$.
Indeed, since any $A\in\Psib^m(X;\Lambda X)$ defines a map
$A:\ddist(X;\Lambda X)\to\ddist(X;\Lambda X)$, our
definition of the wave front set makes sense for $m<0$ as well; it
is independent of $s$ if we take $u\in H^{1,s}_{\bl,\loc}(X;\Lambda X)$ since
the action of $\Psib(X;\Lambda X)$ is well-defined on the larger spaces
$\ddist(X;\Lambda X)$ already.

\begin{Def}\label{Def:WFb-neg}
Suppose $u\in H^{1,s}_{\bl,\loc}(X;\Lambda X)$ for some $s\leq
0$, and suppose that $m\in\RR$.
We say that $q\in\Tb^*X\setminus o$ is not in $\WFb^{1,m}(u)$ if
there exists $A\in\Psib^m(X;\Lambda X)$ such that
$\sigma_{\bl,m}(A)(q)$ is invertible
and $Au\in H^1(X;\Lambda X)$.

For $m=\infty$, we say that $q\in\Tb^*X\setminus o$ is not in $\WFb^{1,m}(u)$
if there exists $A\in\Psib^0(X;\Lambda X)$
such that $\sigma_{\bl,0}(A)(q)$ is invertible
and $LAu\in H^1(X;\Lambda X)$ for all $L\in\Diffb(X;\Lambda X)$, i.e.\ if
$Au\in H^{1,\infty}_\bl(X;\Lambda X)$.

The wave front set $\WFb^{-1,m}(u)$ relative to $\dot H^{-1}(X;\Lambda X)$
is defined similarly for $u\in \dot H^{-1,s}_{\bl,\loc}(X;\Lambda X)$,
and the same also holds for $u\in \dot H^{-1,s}_{\bl,R,\loc}(X;\Lambda X)$
except that we must require $A$ such that $A^*$ satisfies
\eqref{eq:preserve-normal-forms}.
\end{Def}

\begin{rem}
When $A$ as in the definition of $\WFb^{1,m}(u)$ exists,
there also exists $A\in\Psib^m(X;\Lambda X)$ which is
elliptic at
$q$ and which satisfies \eqref{eq:preserve-normal-forms}.
Indeed, one may arrange that $A$ is supported in a local coordinate chart
and is scalar in the trivialization of $\Lambda X$.
\end{rem}

With this background we have the following strengthening of
Theorem~\ref{thm:prop-sing-restate}.

\begin{thm}[Negative order version of Theorem~\ref{thm:prop-sing-restate}]
\label{thm:prop-sing-restate-neg}
Suppose that $P$ is as in \eqref{eq:P-form} considered
as a map \eqref{eq:P-form-neg}, i.e.\ consider $P$
with relative boundary conditions, $Pu=f$.

If $u\in H^{1,m}_{\bl,R,\loc}(X;\Lambda X)$ for some $m\leq 0$,
then for all $s\in\Real\cup\{\infty\}$,
$$
\WFb^{1,s}(u)\setminus\WFb^{-1,s+1}(f)\subset\dot\Sigma,
$$
and it
is a union of maximally extended
generalized broken bicharacteristics of $P$ in
$\dot\Sigma\setminus\WFb^{-1,s+1}(Pu)$.

The same conclusion holds with relative boundary conditions replaced by
absolute boundary conditions.
\end{thm}

\begin{proof}
As noted in \cite[Remark~8.3]{Vasy:Propagation-Wave}, the only
part of our estimates that need changing is the treatment of the
`background terms', such as $\|u\|_{H^1_{\loc}}$ in
Lemma~\ref{lemma:Dirichlet-form} (and Lemma~\ref{lemma:Dt-Dx}),
and $\|Pu\|_{H^{-1}_{R,\loc}(X)}$. Explicitly, we need to
replace the
$$
H^1_{\loc}(X;\Lambda X),\ \text{resp.}
\ H^{-1}_{R,\loc}(X;\Lambda X),
$$
norms by
the
$$
H^{1,m}_{\bl,\loc}(X;\Lambda X),
\ \text{resp.}\ H^{-1,m}_{\bl,R,\loc}(X;\Lambda X),
$$
norms.
The microlocal norms,
in which we are gaining regularity, such as those of $Gu$ and $\tilde G Pu$
in Lemma~\ref{lemma:Dirichlet-form} and Lemma~\ref{lemma:Dt-Dx}
are {\em unchanged.} Indeed, now
we merely need to apply
\cite[Lemma~3.18]{Vasy:Propagation-Wave} in place of
\cite[Lemma~3.13]{Vasy:Propagation-Wave}.
\end{proof}

\section{Other bundles and boundary conditions}\label{sec:other}
In this section we briefly discuss other problems for which our
methods work.

The important ingredients
are the following:
\begin{enumerate}
\item
Suppose $E$ is a vector bundle over a manifold with corners $X$.
Equip $E$ with a Hermitian inner product, $(.,.)_{\tilde H}$, and
equip $X$ with a Lorentzian metric $h$. Assume that every proper boundary
face $F$ of $X$ is time-like, i.e.\ the dual metric $H$ corresponding to $h$
restricts to be negative definite on $N^*F$.
\item
Assume that for each boundary hypersurface
$\hsf$, $E^{\hsf}$ is a subbundle of $E|_{\hsf}$, $\cE=\{E_{\hsf}:\ \hsf\in
\pa_1(X)\}$. Let
$$
\CI_{\cE}(X;E)=\{u\in\CI(X;E):\ \forall \hsf\in\pa_1(X),
\ u|_\hsf\in\CI(X;E^{\hsf})\},
$$
and define $H^1_{\cE}(X;E)$ similarly.
\item
Suppose that the boundary conditions are locally trivializable, i.e.\ for
each $q_0\in \pa X$, with $\hsf_j$, $j=1,\ldots,k$ the boundary
hypersurfaces through $q_0$, there exists a neighborhood $\cU$ of $q_0$ in $X$,
a trivialization of $E|_{\cU}$,
$E|_{\cU}\to \cU\times \RR^N$, $N=\dim E_{q_0}$, and
index sets $J_j\subset\{1,\ldots,N\}$
for $j=1,\ldots,k$,
such that for each $j$ and
at each $q\in \cU\cap \hsf_j$, and for each $\alpha\in E_q$,
$$
\alpha\in E^{\hsf_j}_q\ \text{if and only if}\ \alpha_m=0\ \text{for all}
\ m\in J_j,
$$
$m\in J_j$, where $\alpha=(\alpha_1,\ldots,\alpha_N)$ with respect
to the trivialization (see \eqref{eq:form-triv-corner} for the concrete
case of differential forms).
\item
Consider
$$
\nabla\in\Diff^1(X;E,E\otimes T^*X)
$$
with principal symbol $\sigma_1(\nabla)=\imath\Id\otimes\xi$
(cf.\ \eqref{eq:nabla-symbol}). The inner product on
$E\otimes T^*X$ is given by $\tilde H\otimes H$, $H$ the dual metric
of $h$, cf.\ \eqref{eq:twisted-inner-product},
and is thus not positive definite.
\item
Let $F$ be another vector bundle over $X$, with a non-degenerate
(but not necessarily positive definite) inner product. Then
first order differential operators $Q\in\Diff^1(X;E,F)$ give rise
to maps $Q:H^1_{\cE}(X;E)\to L^2(X;F)$. We thus obtain an adjoint
$$
Q^*:L^2(X;E)=L^2(X;F)^*\to (H^1_{\cE}(X;E))^*\equiv\dot H^{-1}_{\cE}(X;E).
$$
Let $\Diffd^1(X;F,E)$ denote the set of these Banach space adjoints, and
let $\Diffs^2(X;E)$ denote the set of operators of the form
$$
\sum Q_j P_j,\ Q_j\in\Diffd^1(X;F,E),\ P_j\in\Diff^1(X;E,F),
$$
with the sum finite if $X$ is compact, locally finite in general;
cf.\ \eqref{eq:Def-Diffs}. Here $F$ is not important when discussing
$\Diffs^2(X;E)$; one could replace $F$ by the trivial bundle at the
cost of dealing with additional locally finite sums; in particular
the inner product on $F$ is a matter of convenience to state the
form of $P$ below using $\nabla$.
\item
Assume that $P\in\Diffs^2(X;E)$ is of the form
\begin{equation}\label{eq:P-form-general}
P=\nabla^*\nabla+P_1,\ P_1\in\Diff^1(X;E)+\Diffd^1(X;E),
\end{equation}
cf.\ \eqref{eq:P-form}. Note that the principal symbol of $P$ in the
standard sense is $p\,\Id$, where $p$ is the metric function of the
dual Lorentzian metric $H$.
\item
Consider solutions of $Pu=f$ in $u\in H^1_{\cE,\loc}(X;E)$; here
$f\in\dot H^{-1}_{\cE,\loc}(X;E)$ is given.
\end{enumerate}

Under these assumptions, using essentially the same arguments as in the
previous sections,
we have the following analogue of
Theorem~\ref{thm:prop-sing-restate}:

\begin{thm}
\label{thm:prop-sing-general}
Suppose that $P$ is as in \eqref{eq:P-form-general}.
If $u\in H^{1}_{\cE,\loc}(X;E)$,
then for all $s\in\RR^+\cup\{\infty\}$ (with the convention $\infty+1=\infty$),
$$
\WFb^{1,s}(u)\setminus\WFb^{-1,s+1}(f)\subset\dot\Sigma,
$$
and it
is a union of maximally extended
generalized broken bicharacteristics of $P$ in
$\dot\Sigma\setminus\WFb^{-1,s+1}(Pu)$.
\end{thm}

The analogue of Theorem~\ref{thm:prop-sing-restate-neg} also holds.

\def\cprime{$'$} \def\cprime{$'$}

\end{document}